\documentclass[11pt,a4paper]{article}
\usepackage[margin=2.5cm]{geometry}
\usepackage{lmodern}
\usepackage{array}
\usepackage[raggedright]{titlesec}

\usepackage{amssymb}
\usepackage{amsmath}
\usepackage{amsthm}
\usepackage{mathrsfs}
\usepackage{url}
\usepackage{float}
\usepackage[T1]{fontenc}
\usepackage[utf8]{inputenc}
\usepackage{enumitem}

\newcommand{\dx}{\,\mathrm{d}x}

\numberwithin{equation}{section}
\newtheorem{theorem}{Theorem}[section]
\newtheorem{lemma}[theorem]{Lemma}
\newtheorem{corollary}[theorem]{Corollary}

\newtheorem{proposition}[theorem]{Proposition}
\newtheorem{remark}[theorem]{Remark}

\allowdisplaybreaks

\newcommand{\R}{\mathbb{R}}
\newcommand{\QT}{Q_{\mathcal{T}}}
\newcommand{\norm}[1]{\left\|#1\right\|}
\newcommand{\ppos}[1]{\left[#1\right]^{+}}
\newcommand{\pneg}[1]{\left[#1\right]^{-}}
\newcommand{\ptn}{\partial_t}

\usepackage{graphicx}
\usepackage[section]{placeins}
\usepackage[hidelinks,bookmarks=false,hypertexnames=false]{hyperref}
\begin{document}

\title{Spectral stability thresholds for a thermo--reaction--phase system: enthalpy selection and heat exchange}
\author{Gossrin Jean-Marc Bomisso \and Ali Ouattara Kouma \and Marie Esther Anass\'e}
\date{}
\maketitle
\begin{center}\small Universit\'e Nangui Abrogoua, UFR Sciences Fondamentales et Appliqu\'ees\\ 02 BP 801 Abidjan 02, C\^ote d'Ivoire\\[2pt] \texttt{bogojm@yahoo.fr}, \texttt{benzeus96@gmail.com}, \texttt{hadassamarie43@gmail.com}\end{center}
\begin{abstract}
We study a phenomenological nonlinear parabolic model, motivated by thermochromic inks, in which the latent heat of a reversible colour change provides thermal feedback. The model couples heat transport, reversible Arrhenius kinetics and an Allen--Cahn type phase variable. We determine when this feedback destabilises a homogeneous equilibrium and which spatial mode first loses stability. For chemical--phase coupling $\lambda>1/4$, we establish exact spectral thresholds on all Neumann modes. In the isolated problem, conserved enthalpy selects the limiting equilibrium. If the recolouration activation energy does not exceed that of decolouration, the equilibrium is locally stable for every fixed thermal coupling. Otherwise, provided at least one latent heat coefficient is positive, an explicit threshold separates local stability from instability, and the first loss of stability is homogeneous. Under a nondegeneracy condition, we prove a transcritical exchange of stability and algebraic relaxation on the attracting side of the critical centre manifold. For $E_r>E_d$ and $L_c+L_\phi>0$, the threshold with heat exchange is obtained by minimising explicit modal thresholds determined by polynomials of degree at most two. At fixed equilibrium, it increases strictly with the exchange coefficient $\chi$ and lies between two positive multiples of $1+\chi$. At the threshold, at least one conjugate pair of eigenvalues crosses the imaginary axis transversally, and for strong exchange an explicit condition makes the first unstable mode nonhomogeneous. Simulations illustrate the predicted growth rates and critical relaxation.
\end{abstract}

\bigskip

\section{Introduction}
	
	Thermochromic materials change colour with temperature. In the model considered here, the latent heat of this change provides thermal feedback. Two questions follow. When does this feedback destabilise the colour equilibrium? Which spatial mode first loses stability? We answer both questions for a phenomenological reaction--diffusion--phase model motivated by leuco dye inks, explicitly in the parameters. If the activation energy of recolouration does not exceed that of decolouration, the colour equilibrium is locally stable for every fixed thermal coupling. Otherwise, provided at least one latent heat coefficient is positive, an explicit critical coupling separates local stability from instability. In a thermally insulated sample, the first instability is homogeneous and, under a nondegeneracy condition, is accompanied by a transcritical exchange of stability. In a sheet that exchanges heat with the surrounding air, the critical coupling is larger at fixed equilibrium. Stability is lost through a pair of purely imaginary eigenvalues, and the first unstable mode may be spatially nonhomogeneous.

	Seeboth et al.~\cite{SeebothEtal2014} review the mechanisms of thermochromism in polymers, and MacLaren and White~\cite{MacLarenWhite2005} give design rules for reversible leuco dye mixtures. Recent developments in thermochromic security inks, including durability, formulation and printing, are reviewed by Oliveira et al.~\cite{OliveiraEtal2026}.

A practical motivation is the selective disappearance of handwriting in a sunlit parked vehicle, where substantial temperature increases have been documented~\cite{McLarenNullQuinn2005,GrundsteinMeentemeyerDowd2009}. Pilot reports disappearance of FriXion ink above approximately $60\,^{\circ}$C and possible recolouration below $-10\,^{\circ}$C~\cite{PilotFriXionFAQ}. Reversible leuco dye composites exhibit colour hysteresis~\cite{KulcarEtal2010}. Ageing-induced degradation of conventional ballpoint dyes has also been documented~\cite{MarinovicEtal2026}. The model is phenomenological, and its coefficients are generic. In leuco dye composites, colour changes involve molecular rearrangement controlled by the phase transition of a co-solvent~\cite{BasnecEtal2018}. Dinda et al.~\cite{DindaEtal2025} demonstrate the reverse colour response, from colourless to coloured on heating, using leuco dyes, colour developers and fatty-acid phase-change materials. Both responses motivate the description of colour and phase transformations through a chemical fraction and a structural order parameter.

Thermal feedback provides a mechanism for controlling pattern formation in reaction--diffusion systems. Serna, Mu\~nuzuri and Barrag\'an~\cite{SernaMunuzuriBarragan2017} studied Turing patterns in a reversible non-isothermal Gray--Scott model with an energy balance. Van Gorder~\cite{VanGorder2020} analysed the influence of temperature on the conditions for diffusion-driven instability. In the CDIMA and Schnakenberg systems, Dutta and Ghosh~\cite{DuttaGhosh2026} investigated Arrhenius thermal feedback, its effect on dispersion relations and wavelength selection, and the changes in dynamics induced by fixed-temperature boundaries. General criteria for Turing and wave instabilities with diagonal diffusion were established by Villar-Sep\'ulveda and Champneys~\cite{VillarSepulvedaChampneys2023}. Their extension with Krause~\cite{VillarSepulvedaChampneysKrause2025} addresses the design of reaction--cross-diffusion systems with prescribed unstable wavelengths. These studies provide a framework for investigating the stability of spatial modes. Choi et al.~\cite{ChoiEtal2025} combine Turing instability with dynamic bifurcation for the one-dimensional Gray--Scott model and describe the bifurcated stable patterns. Hilder and Kuehn~\cite{HilderKuehn2025} analyse pattern formation and nonlinear waves near a 1:1 resonant Turing and Turing--Hopf instability through amplitude equations. Here we determine the critical intensity of the latent thermal coupling for a fixed thermo--reaction--phase system, both under enthalpy conservation and in the presence of volumetric heat exchange.

	Let $\Omega \subset \R^N$ ($N \geq 1$) be a bounded, connected domain with boundary $\partial\Omega$ of class $C^{1,1}$, and let $\mathcal{T} > 0$ be a final time. We denote $\QT = \Omega \times (0,\mathcal{T})$. The system studied in this work couples three unknowns, namely the absolute temperature $\theta$ (in kelvin), the active chemical fraction $c$ representing the proportion of the coloured form of the material, and the phase variable $\phi$, a structural order parameter. These three unknowns are governed by
	\begin{equation}
		\label{eq:sys}
		\left\{
		\begin{aligned}
			\rho c_p\,\dfrac{\partial\theta}{\partial t} - \nabla\cdot\bigl(k(\phi)\nabla\theta\bigr)
			&= H_{\mathrm{ext}}(x,t) + \alpha L_c\,\dfrac{\partial c}{\partial t} + \alpha L_\phi\,\dfrac{\partial\phi}{\partial t}
			&& \text{in } \QT, \\[4pt]
			\dfrac{\partial c}{\partial t} - d_c\,\Delta c
			&= -\beta K_d(\theta)\,c + \gamma K_r(\theta)\,(1-c)
			&& \text{in } \QT, \\[4pt]
			\tau_\phi\,\dfrac{\partial\phi}{\partial t} - \varepsilon^2\Delta\phi + F'(\phi)
			&= \lambda(c-\phi)
			&& \text{in } \QT,
		\end{aligned}
		\right.
	\end{equation}
	complemented by the homogeneous Neumann boundary conditions
	\begin{equation}
		\label{eq:Neumann}
		\frac{\partial\theta}{\partial\mathbf{n}} = 0,\quad
		\frac{\partial c}{\partial\mathbf{n}} = 0,\quad
		\frac{\partial\phi}{\partial\mathbf{n}} = 0,
		\quad\text{on } \partial\Omega\times(0,\mathcal{T}),
	\end{equation}
	and by the initial conditions
	\begin{equation}
		\label{eq:CI}
		\theta(\cdot,0) = \theta_0,\quad
		c(\cdot,0) = c_0,\quad
		\phi(\cdot,0) = \phi_0,
		\quad\text{in } \Omega.
	\end{equation}
	We denote by $\mathcal R(\theta,c)=-\beta K_d(\theta)c+\gamma K_r(\theta)(1-c)$ the reaction rate of the chemical equation.
	The model parameters are as follows: $\rho$ denotes the mass density, $c_p$ the specific heat capacity, $k(\phi)$ the phase-dependent thermal conductivity, $H_{\mathrm{ext}}(x,t)$ an external heat source, $L_c$ and $L_\phi$ the latent heat coefficients associated respectively with the chemical transformation and the phase transition, $\alpha$ the thermal coupling parameter, $d_c$ the chemical diffusion coefficient, $\beta$ and $\gamma$ the respective intensities of the decolouration and recolouration mechanisms, $\tau_\phi$ the phase relaxation time, $\varepsilon$ the interface parameter, and $\lambda$ the relaxation coefficient coupling $c$ to $\phi$.
	
	The kinetic rates are assumed to be of Arrhenius type, given by
	\begin{equation}
		\label{eq:Arrhenius}
		K_d(\theta) = A_d \exp\!\left(-\frac{E_d}{R\theta}\right),
		\qquad
		K_r(\theta) = A_r \exp\!\left(-\frac{E_r}{R\theta}\right),
		\quad\text{for } \theta > 0,
	\end{equation}
	where $A_d, A_r > 0$ are pre-exponential factors, $E_d, E_r > 0$ are activation energies, and $R > 0$ is the universal gas constant. For the phase variable, we choose the double-well potential defined by
	\begin{equation}
		\label{eq:F}
		F(\phi) = \tfrac{1}{4}\phi^2(1-\phi)^2, \qquad
		F'(\phi) = \tfrac{1}{2}\phi(1-\phi)(1-2\phi),
	\end{equation}
	which has two minima, at $\phi = 0$ and $\phi = 1$. At a fixed value of $c$, the effective potential is given by
\[
G_c(\phi)=F(\phi)+\frac{\lambda}{2}(\phi-c)^2.
\]
Since $F''(\phi)\geq -1/4$, $G_c$ is strictly convex when $\lambda>1/4$. Thus, the double-well shape of $F$ does not imply two stable phase equilibria in the regime of our main stability theorems.
	
	Throughout the paper, $\rho,c_p,d_c,\tau_\phi,\varepsilon,\lambda,\beta,\gamma>0$ and $\alpha,L_c,L_\phi\ge0$. The Arrhenius rates \eqref{eq:Arrhenius} are extended by $K_d(\theta)=K_r(\theta)=0$ for $\theta\le0$. This extension defines the system for every value of $\theta$, and Lemma~\ref{lem:rates} gives its properties. We assume the following two hypotheses.

	\begin{itemize}
		\item[\textbf{(H1)}] \textit{(Thermal conductivity)}
		$k\in C^1(\mathbb R)$ and there exist two constants $0<k_*\le k^*<\infty$ such that
		$k_*\le k(s)\le k^*$ for all $s\in\mathbb R$.

		\item[\textbf{(H2)}] \textit{(Initial data)}
		$\theta_0\in L^\infty(\Omega)$ with $\theta_0\ge\theta_*>0$ a.e. in $\Omega$, $c_0\in H^1(\Omega)$ with $0\le c_0\le 1$ a.e. in $\Omega$, and $\phi_0\in H^1(\Omega)$ with $0\le\phi_0\le 1$ a.e. in $\Omega$.
	\end{itemize}

Arrhenius kinetics are central in combustion theory~\cite{BebernesEberly1989}. Mass-control techniques for reaction--diffusion systems are surveyed in~\cite{PierreM2010}. For thermodynamically consistent energy--reaction--diffusion systems, Haskovec et al.~\cite{HaskovecEtal2018} establish convergence via entropy production, Mielke and Mittnenzweig~\cite{MielkeMittnenzweig2018} develop vector-valued functional inequalities, and Fischer et al.~\cite{FischerEtal2022} prove global existence using entropy estimates. Liu and Sulzbach~\cite{LiuSulzbach2022} derive a non-isothermal reaction--diffusion model and establish global well-posedness near equilibrium for small data in critical Besov spaces.

The phase equation belongs to the Cahn--Hilliard and Allen--Cahn framework~\cite{CahnHilliard1958,AllenCahn1979,MiranvilleBook2019}. Thermal phase-field couplings include the Caginalp and Penrose--Fife models~\cite{Caginalp1986,PenroseFife1990}, the relaxed Stefan problem of Colli and Sprekels~\cite{ColliSprekels1995}, and nonlinear phase-change systems~\cite{KenmochiNiezgodka1994}. Recent developments include a thermodynamically consistent non-isothermal diffuse-interface model and its sharp-interface asymptotics by Liu, Sulzbach and Wang~\cite{LiuSulzbachWang2025}, and a non-isothermal Allen--Cahn model for tumour growth with local and global well-posedness results by Gatti, Ipocoana and Miranville~\cite{GattiIpocoanaMiranville2026}. General long-time dynamics are treated in~\cite{TemamBook1997,MiranvilleZelik2004}. For thermally coupled phase-field systems, Miranville et al.~\cite{MiranvilleQuintanillaSaoud2020} establish exponential attractors, Makki et al.~\cite{MakkiMiranvilleSaoud2020} obtain strict separation in one space dimension for singular potentials, and Ntsokongo~\cite{Ntsokongo2023} studies well-posedness and attractors for an Allen--Cahn type system.

In enthalpy variables the principal diffusion matrix of \eqref{eq:sys} is triangular, and the thermal coupling enters the modal characteristic polynomial explicitly. The isolated system has a zero eigenvalue associated with the conserved enthalpy, so normal stability is studied along a curve of homogeneous equilibria. For $\lambda>1/4$, the sign of $E_r-E_d$ determines the stability classification. The nonlinear stability and bifurcation results in Section~\ref{sec:local} are local. If $E_r\le E_d$, small perturbations converge exponentially to the equilibrium selected by their enthalpy. If $E_r>E_d$ and $L_c+L_\phi>0$, the exact adiabatic threshold at a homogeneous equilibrium $(\theta_e,c_e,\phi_e)$ is given by
\[
\alpha_{\mathrm{crit}}=\frac{\rho c_p}{c_*'(\theta_e)\bigl(L_c+\lambda L_\phi/(F''(\phi_e)+\lambda)\bigr)},
\]
where $c_*$ is the equilibrium chemical fraction defined in Section~\ref{sec:stab}. The first loss of normal stability is homogeneous and occurs precisely at the degeneracy of the equilibrium enthalpy map. Under an explicit nondegeneracy condition, we compute the quadratic coefficient of the homogeneous centre-manifold equation, establish a transcritical exchange of stability along parameter-dependent enthalpy classes, and obtain the critical algebraic relaxation on the attracting side. The centre-manifold coordinate is normalised by the temperature deviation at critical enthalpy. At fixed coupling, varying the enthalpy yields a fold of the equilibrium branch. In the same parameter regime, the threshold with heat exchange is obtained by minimising explicit modal thresholds over all Neumann modes, each determined by a polynomial of degree at most two, and it lies between two positive multiples of $1+\chi$, where $\chi$ is the volumetric heat exchange coefficient. At least one conjugate pair of eigenvalues crosses the imaginary axis transversally, and for strong exchange an explicit condition makes the first unstable mode nonhomogeneous. Section~\ref{sec:local} assumes in addition a boundary of class $C^{2+\delta}$ and $k\in C^3(\mathbb R)$.

	Table~\ref{tab:parameters} lists the parameters of the model, their units in a dimensional formulation and the way they could be determined. Several of them are phenomenological, and the analysis holds for all admissible values.
	\begin{table}[ht]
	\centering
	\footnotesize
	\setlength{\tabcolsep}{4pt}
	\begin{tabular}{@{}l >{\raggedright\arraybackslash}p{4.2cm} l >{\raggedright\arraybackslash}p{4.2cm}@{}}
	\hline
	Symbol & Meaning & Unit & Determination\\
	\hline
	$a=\rho c_p$ & volumetric heat capacity & J\,m$^{-3}$\,K$^{-1}$ & calorimetry\\
	$k(\phi)$ & thermal conductivity & W\,m$^{-1}$\,K$^{-1}$ & standard measurement\\
	$L_c$, $L_\phi$ & latent heat coefficients & J\,m$^{-3}$ & differential scanning calorimetry\\
	$\alpha$ & thermal coupling intensity & -- & phenomenological\\
	$A_d$, $A_r$ & pre-exponential factors & s$^{-1}$ & kinetics of the colour change\\
	$E_d$, $E_r$ & activation energies & J\,mol$^{-1}$ & Arrhenius plots of colour kinetics\\
	$\beta$, $\gamma$ & rate weights & -- & absorbed in $A_d$, $A_r$\\
	$d_c$ & chemical diffusivity & m$^2$\,s$^{-1}$ & not determined in the present work\\
	$\tau_\phi$, $\varepsilon^2$, $\lambda$ & phase relaxation, interface, coupling & s, m$^2$, -- & phenomenological\\
	$\chi=2h_s/e_s$ & volumetric heat exchange & W\,m$^{-3}$\,K$^{-1}$ & surface coefficient and thickness\\
	$Q=q_r/e_s$ & absorbed radiative power & W\,m$^{-3}$ & absorbed flux $q_r$ and thickness\\
	$\theta_a$ & ambient temperature & K & measured\\
	\hline
	\end{tabular}
	\caption{Parameters of the model. Here $h_s$ is the surface heat transfer coefficient, $e_s$ the thickness of the sheet and $q_r$ the absorbed radiative flux, see Section~\ref{sec:exchange}.}
	\label{tab:parameters}
	\end{table}

	The thresholds depend on the parameters through a small number of dimensionless groups. We fix a homogeneous equilibrium $(\theta_e,c_e,\phi_e)$ with $\theta_e>0$ and use the reference scales $\theta_e$ for the temperature, $t_0=1/b_e$ for the time, with $b_e=\beta K_d(\theta_e)+\gamma K_r(\theta_e)$, and $\ell_0=|\Omega|^{1/N}$ for the length. We introduce the dimensionless groups defined by
	\[
	\begin{gathered}
	\mathrm{Lt}_c=\frac{\alpha L_c}{\rho c_p\theta_e},\qquad
	\mathrm{Lt}_\phi=\frac{\alpha L_\phi}{\rho c_p\theta_e},\qquad
	\bar\chi=\frac{\chi}{\rho c_p b_e},\qquad
	\Gamma_d=\frac{E_d}{R\theta_e},\qquad
	\Gamma_r=\frac{E_r}{R\theta_e},\\
	D_\theta=\frac{k(\phi_e)}{\rho c_p b_e\ell_0^2},\qquad
	D_c=\frac{d_c}{b_e\ell_0^2},\qquad
	T_\phi=\tau_\phi b_e,\qquad
	D_\phi=\frac{\varepsilon^2}{\tau_\phi b_e\ell_0^2}.
	\end{gathered}
	\]
	The numbers $\mathrm{Lt}_c$ and $\mathrm{Lt}_\phi$ are latent heat ratios of the chemical and structural transformations. They compare latent and sensible heat and are analogous to inverse Stefan numbers, with $\theta_e$ chosen as the temperature scale. The dimensionless heat exchange parameter $\bar\chi$ compares the rate of heat exchange with the rate of chemical relaxation, and $\Gamma_d,\Gamma_r$ are reduced activation energies. With $\hat\theta=\theta/\theta_e$, the normalised rates are $\beta K_d(\theta)/b_e=(1-c_e)\exp(\Gamma_d(1-1/\hat\theta))$ and $\gamma K_r(\theta)/b_e=c_e\exp(\Gamma_r(1-1/\hat\theta))$. The derivative of the equilibrium coloured fraction is given by
	\[
	\theta_e\,c_*'(\theta_e)=c_e(1-c_e)(\Gamma_r-\Gamma_d).
	\]
	The modal polynomials of Section~\ref{sec:local} contain $\alpha$, $L_c$, $L_\phi$ and $c_*'(\theta_e)$ only through the products $\theta_ec_*'(\theta_e)\mathrm{Lt}_c$ and $\theta_ec_*'(\theta_e)\mathrm{Lt}_\phi$. In particular, the adiabatic threshold is the condition
	\[
	c_e(1-c_e)(\Gamma_r-\Gamma_d)\Bigl(\mathrm{Lt}_c+\frac{\lambda}{F''(\phi_e)+\lambda}\,\mathrm{Lt}_\phi\Bigr)=1.
	\]
	The equilibrium thus loses stability when the latent heat ratios, weighted by the reduced activation-energy gap, reach order one. With heat exchange, the threshold depends in addition on $\bar\chi$ and on the diffusion numbers $D_\theta$, $D_c$ and $D_\phi$.

	Section~\ref{sec:max} establishes the invariant bounds, Section~\ref{sec:stab} the enthalpy balance and a conditional decay estimate, and Section~\ref{sec:local} local well-posedness, exact thresholds, bifurcation dynamics, mode selection under heat exchange, numerical illustrations and a physical interpretation of the results.

\section{Invariant bounds}
	\label{sec:max}

	The analysis of this section is carried out for a sufficiently regular solution of \eqref{eq:sys}--\eqref{eq:CI} on $[0,\mathcal T]$. The invariance of $[0,1]$ for $c$ and $\phi$ uses only the nonnegativity of the extended rates and the initial bounds in \textbf{(H2)}. It is a scalar form of the invariant region principle of Chueh, Conley and Smoller~\cite{ChuehConleySmoller1977}.

	\subsection{Preliminary tools}
	\label{sec:prelim}

	For a measurable function $u:\Omega\to\R$ we write $\ppos{u}=\max(u,0)$ and $\pneg{u}=\max(-u,0)$, so that $u=\ppos{u}-\pneg{u}$. If $u\in H^1(\Omega)$, then $\pneg{u}\in H^1(\Omega)$ and $\nabla\pneg{u}=-\mathbf{1}_{\{u<0\}}\nabla u$ a.e. in $\Omega$~\cite{Ziemer}.

	\begin{lemma}[Properties of the extended rates]
		\label{lem:rates}
		The extended functions $K_d,K_r$ belong to $C^\infty(\mathbb R)$, are nonnegative and nondecreasing, and satisfy
		\[
		\begin{aligned}
		&0\le K_d\le A_d,\qquad 0\le K_r\le A_r,\\
		&0\le K_d'\le \frac{4RA_d}{e^2E_d},\qquad 0\le K_r'\le \frac{4RA_r}{e^2E_r},
		\end{aligned}
		\quad\text{on }\mathbb R.
		\]
	\end{lemma}

	\begin{proof}
		Let $K(\theta)=A\exp(-E/(R\theta))$ for $\theta>0$ and $K(\theta)=0$ for $\theta\le0$. For $\theta>0$, every derivative of $K$ is a polynomial in $1/\theta$ multiplied by $\exp(-E/(R\theta))$, and tends to $0$ as $\theta\downarrow0$. Hence $K\in C^\infty(\mathbb R)$. Moreover $0\le K\le A$. With $x=E/(R\theta)$,
		\[
		K'(\theta)=\frac{AR}{E}\,x^2e^{-x}\le\frac{4AR}{e^2E},
		\]
		since $x\mapsto x^2e^{-x}$ attains its maximum $4e^{-2}$ at $x=2$. In particular $K'\ge0$.
	\end{proof}

	\subsection{Positivity of the temperature}
	\label{sub:theta}

	\begin{theorem}[A priori lower bound for $\theta$]
		\label{thm:theta_pos}
		Assume \textup{\textbf{(H1)}}--\textup{\textbf{(H2)}}. Let a sufficiently regular solution of \eqref{eq:sys}--\eqref{eq:CI} be such that the total thermal source
		\[
		S = H_{\mathrm{ext}} + \alpha L_c\,\ptn c + \alpha L_\phi\,\ptn\phi
		\]
		satisfies, for some constant $C_0\ge0$,
		\begin{equation}\label{eq:source}
		S\in L^\infty(\QT),\qquad S\ge-C_0,\quad\text{a.e. in }\QT.
		\end{equation}
		Then the temperature satisfies
		\[
		\theta(x,t) \geq \underline{\theta}(t)=\theta_* - \frac{C_0}{\rho c_p}\,t > 0,
		\quad\text{for a.e. } (x,t) \in \Omega\times(0,\min\{\mathcal T,T_0\}),
		\]
		where $T_0=\rho c_p\,\theta_*/C_0$, with $T_0 = +\infty$ if $C_0 = 0$.
	\end{theorem}
	
	\begin{proof}
Let $0<T_1<\min\{\mathcal T,T_0\}$ and set $w=\theta-\underline\theta$ on $\Omega\times(0,T_1)$. Since $\underline\theta$ depends only on $t$ and $\rho c_p\underline\theta'=-C_0$, the function $w$ satisfies
\[
\left\{\begin{aligned}
\rho c_p\frac{\partial w}{\partial t}-\nabla\cdot\bigl(k(\phi)\nabla w\bigr)&=S+C_0\ge0 &&\text{in }\Omega\times(0,T_1),\\
\frac{\partial w}{\partial\mathbf n}&=0&&\text{on }\partial\Omega\times(0,T_1),\\
w(\cdot,0)&=\theta_0-\theta_*\ge0&&\text{in }\Omega.
\end{aligned}\right.
\]
The negative part $\pneg{w}$ belongs to $H^1(\Omega)$ for a.e. $t$, with $\nabla\pneg{w}=-\mathbf 1_{\{w<0\}}\nabla w$~\cite{Ziemer}. We multiply the equation by $-\pneg{w}$ and integrate over $\Omega$. Since $w\,\pneg{w}=-(\pneg{w})^2$ and $\nabla w\cdot\nabla\pneg{w}=-|\nabla\pneg{w}|^2$, the Neumann condition gives
\[
\frac{\rho c_p}{2}\frac{d}{dt}\|\pneg{w}\|_{L^2}^2+\int_\Omega k(\phi)|\nabla\pneg{w}|^2\,dx=-\int_\Omega(S+C_0)\pneg{w}\,dx\le0.
\]
By \textbf{(H1)}, the second term is bounded below by $k_*\|\nabla\pneg{w}\|_{L^2}^2\ge0$. Hence $t\mapsto\|\pneg{w}(t)\|_{L^2}^2$ is nonincreasing. It vanishes at $t=0$, so $\pneg{w}=0$ on $(0,T_1)$. This gives $\theta\ge\underline\theta(t)$ a.e., and the bound
\begin{equation}\label{eq3.1}
\theta(x,t)\ge\underline\theta(T_1)=\theta_*-\frac{C_0T_1}{\rho c_p}>0,
\quad\text{for a.e. }(x,t)\in\Omega\times(0,T_1).
\end{equation}
Since $T_1<\min\{\mathcal T,T_0\}$ is arbitrary, the theorem follows.
\end{proof}

	\begin{remark}
		\label{rem:theta_global}
		When $C_0 = 0$ (everywhere non-negative total source), the sub-solution $\underline{\theta}(t) = \theta_*$
		is constant, and the lower bound persists on any interval of existence on which $S\geq0$. For $\alpha=0$ one has $S=H_{\mathrm{ext}}$, and \eqref{eq:source} is a condition on the data only. For $\alpha>0$ and a classical solution on $[0,\mathcal T]$, the source $S$ is bounded and \eqref{eq:source} holds with $C_0=\|[S]^-\|_{L^\infty(\QT)}$. The theorem then gives positivity at least up to the explicit time $\rho c_p\theta_*/C_0$. Near a homogeneous equilibrium, positivity holds for all times without \eqref{eq:source}, by Theorem~\ref{thm:local-stability}.
	\end{remark}
	
	\begin{corollary}[Upper bound on the temperature]
		\label{cor:theta_upper}
		Under the assumptions of Theorem~\ref{thm:theta_pos}, the temperature satisfies
		\[
		\theta(x,t)\le \norm{\theta_0}_{L^\infty(\Omega)}+\frac{t}{\rho c_p}\norm{S}_{L^\infty(\QT)},
		\quad\text{for a.e. }(x,t)\in\QT.
		\]
	\end{corollary}
	
	\begin{proof}
Set $\overline\theta(t)=\|\theta_0\|_{L^\infty}+t\|S\|_{L^\infty(\QT)}/(\rho c_p)$ and $w=\overline\theta-\theta$. Then $w(\cdot,0)\ge0$, $\partial_{\mathbf n}w=0$ and
\[
\rho c_p\frac{\partial w}{\partial t}-\nabla\cdot\bigl(k(\phi)\nabla w\bigr)=\|S\|_{L^\infty(\QT)}-S\ge0,
\quad\text{a.e. in }\QT.
\]
Multiplying by $-\pneg{w}$ and integrating over $\Omega$ gives, as in the previous proof,
\[
\frac{\rho c_p}{2}\frac{d}{dt}\|\pneg{w}\|_{L^2}^2+k_*\|\nabla\pneg{w}\|_{L^2}^2\le0.
\]
Hence $\pneg{w}=0$, which is the upper bound.
\end{proof}

	\subsection{Invariance of $[0,1]$ for the chemical fraction $c$}
	\label{sub:c}
	
	\begin{theorem}[Invariance of {$[0,1]$} for $c$]
		\label{thm:c_bounds}
		Under hypotheses \textup{\textbf{(H1)}}--\textup{\textbf{(H2)}}, any sufficiently regular solution $c$ of the
		system~\eqref{eq:sys}-\eqref{eq:CI} satisfies
		\[
		0 \leq c(x,t) \leq 1, \quad\text{a.e. in } \QT.
		\]
	\end{theorem}
	
	\begin{proof}
\emph{Lower bound.} Set $v=\pneg{c}$. On the set $\{c<0\}$ one has $-\beta K_d(\theta)c\ge0$ and $\gamma K_r(\theta)(1-c)\ge0$, by Lemma~\ref{lem:rates}. Hence $\mathcal R(\theta,c)\pneg{c}\ge0$ a.e. We multiply the chemical equation by $-\pneg{c}$ and integrate over $\Omega$. The Neumann condition gives
\[
\frac12\frac{d}{dt}\|v\|_{L^2}^2+d_c\|\nabla v\|_{L^2}^2=-\int_\Omega\mathcal R(\theta,c)\pneg{c}\,dx\le0.
\]
\emph{Upper bound.} Set $v=\ppos{c-1}$. On the set $\{c>1\}$ one has $-\beta K_d(\theta)c\le0$ and $\gamma K_r(\theta)(1-c)\le0$. Hence $\mathcal R(\theta,c)\ppos{c-1}\le0$ a.e. Multiplying the chemical equation by $\ppos{c-1}$ gives
\[
\frac12\frac{d}{dt}\|v\|_{L^2}^2+d_c\|\nabla v\|_{L^2}^2=\int_\Omega\mathcal R(\theta,c)\ppos{c-1}\,dx\le0.
\]
In both cases $\|v(t)\|_{L^2}$ is nonincreasing and $v(0)=0$ by \textbf{(H2)}. Therefore $v=0$ for all $t$, and $0\le c\le1$ a.e. in $\QT$.
\end{proof}

	\begin{remark}
		The proof uses only $K_d,K_r\ge0$, $\beta,\gamma>0$ and $0\le c_0\le1$. Since the extended rates are nonnegative on $\mathbb R$, the bounds hold on the whole interval of existence, independently of the time $T_0$ of Theorem~\ref{thm:theta_pos}.
	\end{remark}
	
	\subsection{Invariance of $[0,1]$ for the phase variable $\phi$}
	\label{sub:phi}
	
	We begin with a lemma on the sign of $F'$ at the boundary of the interval $[0,1]$,
	which is used to handle the term $F'(\phi)$ in the subsequent proofs.
	
	\begin{lemma}[Sign of $F'$ outside {$[0,1]$}]
		\label{lem:F_sign}
		One has $F'(\phi)\le0$ for $\phi<0$ and $F'(\phi)\ge0$ for $\phi>1$. Hence $F'(\phi)\pneg{\phi}\le0$ and $F'(\phi)\ppos{\phi-1}\ge0$ on $\mathbb R$.
	\end{lemma}

	\begin{proof}
The derivative of $F$ is given by
\[
F'(\phi)=\tfrac12\,\phi\,(1-\phi)\,(1-2\phi).
\]
For $\phi<0$, the three factors $\phi$, $1-\phi$ and $1-2\phi$ have the signs $-,+,+$, so $F'(\phi)<0$. For $\phi>1$, they have the signs $+,-,-$, so $F'(\phi)>0$. Since $\pneg{\phi}$ vanishes for $\phi\ge0$ and $\ppos{\phi-1}$ vanishes for $\phi\le1$, the products $F'(\phi)\pneg{\phi}\le0$ and $F'(\phi)\ppos{\phi-1}\ge0$ hold on $\mathbb R$.
\end{proof}

	\begin{theorem}[Invariance of {$[0,1]$} for $\phi$]
		\label{thm:phi_bounds}
		Under hypotheses \textup{\textbf{(H1)}}--\textup{\textbf{(H2)}}, any sufficiently regular solution $\phi$
		of the system~\eqref{eq:sys}-\eqref{eq:CI} satisfies
		\[
		0 \leq \phi(x,t) \leq 1, \quad\text{a.e. in } \QT.
		\]
	\end{theorem}
	
	\begin{proof}
\emph{Lower bound.} Set $v=\pneg{\phi}$. We multiply the phase equation by $-\pneg{\phi}$ and integrate over $\Omega$. Since $\phi\,\pneg{\phi}=-v^2$, the coupling term gives
\[
-\lambda\int_\Omega(c-\phi)\pneg{\phi}\,dx=-\lambda\int_\Omega c\,v\,dx-\lambda\|v\|_{L^2}^2\le-\lambda\|v\|_{L^2}^2,
\]
because $c\ge0$ by Theorem~\ref{thm:c_bounds}. By Lemma~\ref{lem:F_sign}, $-\int_\Omega F'(\phi)\pneg{\phi}\,dx\ge0$. With the Neumann condition, this yields
\[
\frac{\tau_\phi}{2}\frac{d}{dt}\|v\|_{L^2}^2+\varepsilon^2\|\nabla v\|_{L^2}^2+\lambda\|v\|_{L^2}^2\le0.
\]
\emph{Upper bound.} Set $v=\ppos{\phi-1}$ and multiply the phase equation by $v$. Since $c-\phi=(c-1)-(\phi-1)$ and $c\le1$, one has
\[
\lambda\int_\Omega(c-\phi)v\,dx\le-\lambda\|v\|_{L^2}^2.
\]
By Lemma~\ref{lem:F_sign}, $\int_\Omega F'(\phi)v\,dx\ge0$, and the same inequality follows. In both cases $v(0)=0$ by \textbf{(H2)}, hence $v=0$ for all $t$, and $0\le\phi\le1$ a.e. in $\QT$.
\end{proof}

	\section{Enthalpy conservation and stability estimates}
	\label{sec:stab}

	In this section $H_{\mathrm{ext}}=0$ and the solution is assumed to be sufficiently regular to justify the integrations by parts below. The conservation law is valid without a smallness assumption on $\alpha$.

	\subsection{The conserved enthalpy}

	We define the enthalpy by
	\begin{equation}\label{eq:enthalpy}
	 h=\rho c_p\theta-\alpha L_c c-\alpha L_\phi\phi.
	\end{equation}

	\begin{proposition}[Enthalpy balance]\label{prop:enthalpy}
	 For every sufficiently regular solution of \eqref{eq:sys}--\eqref{eq:CI},
	 \[
	 \frac{d}{dt}\int_\Omega h(x,t)\,dx=\int_\Omega H_{\mathrm{ext}}(x,t)\,dx.
	 \]
	 In particular, when $H_{\mathrm{ext}}=0$ the mean enthalpy is conserved.
	\end{proposition}
	\begin{proof}
	 Subtract $\alpha L_c\partial_t c+\alpha L_\phi\partial_t\phi$ from both sides of the heat equation. This gives
	 \[
	 \partial_t h-\nabla\cdot(k(\phi)\nabla\theta)=H_{\mathrm{ext}}.
	 \]
	 Integrating over $\Omega$ and applying the homogeneous Neumann condition yields the asserted balance.
	\end{proof}

	\begin{remark}[Enthalpy constraint on the limit equilibrium]\label{rem:enthalpy-compatibility}
	The mean of $\theta$ is not conserved in general, even when $H_{\mathrm{ext}}=0$, because it changes through the evolution of $c$ and $\phi$. For instance, if $\alpha L_c>0$, $\theta_0=\bar\theta$, $\phi_0=\bar\phi$ and $c_0=\bar c+\xi_0\in(0,1)$ with $\xi_0\ne0$, the initial temperature has the mean $\bar\theta$ but the enthalpy differs from that of $(\bar\theta,\bar c,\bar\phi)$ by $-\alpha L_c\xi_0|\Omega|\ne0$. No trajectory with these data can converge to $(\bar\theta,\bar c,\bar\phi)$ in $L^2(\Omega)^3$. The appropriate compatibility condition is \eqref{eq:compat} below.
	\end{remark}

	\subsection{Homogeneous equilibria and their enthalpy}
	Write
	\[
	 a=\rho c_p,\qquad A=\alpha L_c,\qquad B=\alpha L_\phi,
	\qquad b(\theta)=\beta K_d(\theta)+\gamma K_r(\theta).
	\]
	For $\theta>0$, $b(\theta)>0$ for the Arrhenius rates \eqref{eq:Arrhenius}. We define
	\[
	 c_*(\theta)=\frac{\gamma K_r(\theta)}{b(\theta)}.
	\]
	Then $c_*(\theta)\in(0,1)$ and the chemical equilibrium equation is equivalent to $c=c_*(\theta)$. A homogeneous equilibrium is a triple $(\bar\theta,\bar c,\bar\phi)$ with $\bar\theta>0$, $\bar c=c_*(\bar\theta)$ and $F'(\bar\phi)=\lambda(\bar c-\bar\phi)$.
	For $\lambda>1/4$, every $\bar\theta>0$ defines exactly one homogeneous equilibrium, and $\bar\phi\in(0,1)$. In fact, $F''(s)=3(s-1/2)^2-1/4$, so $s\mapsto F'(s)+\lambda s$ is strictly increasing on $\mathbb R$. It takes the value $0<\lambda\bar c$ at $s=0$ and the value $\lambda>\lambda\bar c$ at $s=1$.

	\begin{proposition}[Local equilibrium branch]\label{prop:equilibrium_branch}
	 Let $(\bar\theta,\bar c,\bar\phi)$ be a homogeneous equilibrium and assume
	 $F''(\bar\phi)+\lambda\ne0$. There are an interval $I$ containing $\bar\theta$ and a unique $C^1$ function $\Phi:I\to\mathbb R$ satisfying
	 \[
	 \Phi(\bar\theta)=\bar\phi,\qquad
	 F'(\Phi(\theta))+\lambda\Phi(\theta)=\lambda c_*(\theta).
	 \]
	 Every homogeneous equilibrium sufficiently close to $(\bar\theta,\bar c,\bar\phi)$ has the form
	 $(\theta,c_*(\theta),\Phi(\theta))$, with $\theta\in I$.
	 Put
	 \begin{equation}\label{eq:Hcurve}
	 \mathscr H(\theta)=a\theta-Ac_*(\theta)-B\Phi(\theta).
	 \end{equation}
	 If $\mathscr H'(\bar\theta)\ne0$, the conserved mean enthalpy selects exactly one nearby homogeneous equilibrium. In particular, $\mathscr H'(\bar\theta)>0$ for all sufficiently small $\alpha\ge0$, with all other parameters fixed.
	\end{proposition}
	\begin{proof}
	 The first equilibrium equation is linear in $c$ with coefficient $-b(\theta)$, hence $c=c_*(\theta)$. Define
	 $Q(\theta,s)=F'(s)+\lambda s-\lambda c_*(\theta)$.
	 Since $Q(\bar\theta,\bar\phi)=0$ and
	 $\partial_s Q(\bar\theta,\bar\phi)=F''(\bar\phi)+\lambda\ne0$, the implicit-function theorem gives $\Phi$. Differentiation yields
	 \[
	 \Phi'(\theta)=\frac{\lambda c_*'(\theta)}{F''(\Phi(\theta))+\lambda}.
	 \]
	 Consequently, the derivative of $\mathscr H$ is given by
	 \[
	 \mathscr H'(\theta)=a-Ac_*'(\theta)-B\Phi'(\theta).
	 \]
	 The one-dimensional inverse-function theorem proves local uniqueness at prescribed enthalpy whenever $\mathscr H'(\bar\theta)\ne0$. As $A=\alpha L_c$ and $B=\alpha L_\phi$, one has $\mathscr H'(\bar\theta)\to a>0$ as $\alpha\downarrow0$.
	\end{proof}

	\begin{lemma}[Characterisation of stationary solutions]
	\label{lem:stationary-completeness}
	Assume \textbf{(H1)}, $H_{\mathrm{ext}}=0$ and $\lambda>1/4$.
	Every sufficiently regular stationary solution of \eqref{eq:sys}--\eqref{eq:Neumann}
	with $\theta>0$ is homogeneous. Consequently, in a neighbourhood of a reference
	equilibrium satisfying Proposition~\ref{prop:equilibrium_branch}, the complete
	set of stationary solutions coincides with the curve
	$\theta\mapsto(\theta,c_*(\theta),\Phi(\theta))$.
	\end{lemma}
	\begin{proof}
	At equilibrium the thermal equation reads
	$-\nabla\cdot(k(\phi)\nabla\theta)=0$.
	Testing by $\theta$ and using the Neumann condition gives
	$\int_\Omega k(\phi)|\nabla\theta|^2\,dx=0$.
	Since $k\ge k_*>0$ and $\Omega$ is connected, $\theta$ is constant.
	Thus $b(\theta)>0$ is constant and the chemical equation becomes
	$-d_c\Delta(c-c_*(\theta))+b(\theta)(c-c_*(\theta))=0$.
	Testing by $c-c_*(\theta)$ shows that $c=c_*(\theta)$.
	Finally, the stationary phase equation is
	\[
	-\varepsilon^2\Delta\phi+F'(\phi)+\lambda\phi
	=\lambda c_*(\theta).
	\]
	The scalar function $s\mapsto F'(s)+\lambda s$ is strictly increasing on
	$\mathbb R$, since $F''(s)+\lambda\ge\lambda-1/4>0$.
	The constant root $\Phi(\theta)$ of
	$F'(s)+\lambda s=\lambda c_*(\theta)$ exists uniquely, because the
	left-hand side tends to $\pm\infty$ as $s\to\pm\infty$.
	Testing the difference of the stationary equation and this constant
	equation by $\phi-\Phi(\theta)$ yields
	\[
	\varepsilon^2\|\nabla(\phi-\Phi(\theta))\|_{L^2}^2
	+(\lambda-1/4)\|\phi-\Phi(\theta)\|_{L^2}^2\le0.
	\]
	Hence $\phi=\Phi(\theta)$, and the stationary solution is homogeneous.
\end{proof}

	\subsection{An exponential estimate on the compatible enthalpy class}

	We write
	\[
	 u=\theta-\bar\theta,\quad v=c-\bar c,\quad
	 w=\phi-\bar\phi,\quad z=au-Av-Bw.
	\]
	Let $(\bar\theta,\bar c,\bar\phi)$ be a homogeneous equilibrium. We assume the enthalpy compatibility condition
	\begin{equation}\label{eq:compat}
	\int_\Omega\bigl[\rho c_p(\theta_0-\bar\theta)
	-\alpha L_c(c_0-\bar c)-\alpha L_\phi(\phi_0-\bar\phi)\bigr]\,dx=0.
	\end{equation}
	By Proposition~\ref{prop:enthalpy} and \eqref{eq:compat},
	\begin{equation}\label{eq:mean-z}
	\int_\Omega z(x,t)\,dx=0
	\end{equation}
	for every time for which the solution exists. Denote by $C_P$ the constant in the Neumann--Poincar\'e inequality
	$\|q-|\Omega|^{-1}\int_\Omega q\|_{L^2}\le C_P\|\nabla q\|_{L^2}$.

	\begin{theorem}[Exponential decay]\label{thm:conditional_decay}
	 Assume \textbf{(H1)}--\textbf{(H2)}, $H_{\mathrm{ext}}=0$, \eqref{eq:compat} and $\lambda>1/2$. Let a sufficiently regular solution be defined on $[0,T]$, and let $0<\theta_-\le\bar\theta$ satisfy
	 \begin{equation}\label{eq:decay-region}
	 \theta(x,t)\ge\theta_-,
	 \quad\text{for a.e. }(x,t)\in\Omega\times(0,T).
	 \end{equation}
	 Define
	 \[
	 b_0=b(\theta_-)=\beta K_d(\theta_-)+\gamma K_r(\theta_-),\qquad
	 M=\frac{4R}{e^2}\Bigl(\frac{\beta\bar cA_d}{E_d}+\frac{\gamma(1-\bar c)A_r}{E_r}\Bigr).
	 \]
	 Suppose there exist weights $\mu,\nu>0$ satisfying
	 \begin{equation}\label{eq:smallcoupling}
	 \begin{aligned}
	 &\frac{\mu M^2C_P^2}{2b_0}\le\frac{ak_*}{4},\\
	 &\frac{(k^*)^2A^2}{ak_*}\le\frac{\mu d_c}{2},\qquad
	 \frac{(k^*)^2B^2}{ak_*}\le\frac{\nu\varepsilon^2}{2},\\
	 &\frac{\nu\lambda}{2}+\frac{\mu M^2A^2}{b_0a^2}\le\frac{\mu b_0}{4},\\
	 &\frac{\mu M^2B^2}{b_0a^2}\le\frac{\nu(2\lambda-1)}{8}.
	 \end{aligned}
	 \end{equation}
	 Then there are $C,\varsigma>0$, independent of $t$ and $T$, such that
	 \begin{equation}\label{eq:conditional_decay}
	 \begin{aligned}
	 &\|u(t)\|_{L^2}^2+\|v(t)\|_{L^2}^2+\|w(t)\|_{L^2}^2\\
	 &\qquad\le C e^{-\varsigma t}
	 \bigl(\|u(0)\|_{L^2}^2+\|v(0)\|_{L^2}^2+\|w(0)\|_{L^2}^2\bigr),
	 \quad\text{for } t\in[0,T].
	 \end{aligned}
	 \end{equation}
	 For a fixed equilibrium and a fixed $\theta_-$, conditions \eqref{eq:smallcoupling} hold for a suitable choice of $\mu,\nu$ when $\alpha\ge0$ is sufficiently small. If the solution exists and satisfies \eqref{eq:decay-region} for all $t\ge0$, estimate \eqref{eq:conditional_decay} holds for all $t\ge0$.
	\end{theorem}

	\begin{proof}
	 Set $\|\cdot\|=\|\cdot\|_{L^2(\Omega)}$.
	 The enthalpy perturbation satisfies
	 \[
	 \partial_tz=\nabla\cdot(k(\phi)\nabla u).
	 \]
	 Testing by $z$ and integrating by parts gives
	 \[
	 \begin{aligned}
	 &\frac12\frac{d}{dt}\|z\|^2+a\int_\Omega k(\phi)|\nabla u|^2\,dx\\
	 &\qquad=A\int_\Omega k(\phi)\nabla u\cdot\nabla v\,dx
	 +B\int_\Omega k(\phi)\nabla u\cdot\nabla w\,dx.
	 \end{aligned}
	 \]
	 Using $k_*\le k(\phi)\le k^*$ and Young's inequality, we obtain
	 \begin{equation}\label{eq:z-diss}
	 \frac12\frac{d}{dt}\|z\|^2+\frac{ak_*}{2}\|\nabla u\|^2
	 \le\frac{(k^*)^2}{ak_*}\bigl(A^2\|\nabla v\|^2+B^2\|\nabla w\|^2\bigr).
	 \end{equation}
	 Write $\mathcal R(\theta,c)-\mathcal R(\bar\theta,\bar c)
	 =-b(\theta)v+\mathfrak r(\theta)$, where
	 $\mathfrak r(\theta)=\mathcal R(\theta,\bar c)-\mathcal R(\bar\theta,\bar c)$.
	 Since $\partial_\theta\mathcal R(s,\bar c)=-\beta\bar cK_d'(s)+\gamma(1-\bar c)K_r'(s)$, the bounds on $K_d',K_r'$ in Lemma~\ref{lem:rates} and the mean value theorem give $|\mathfrak r(\theta)|\le M|u|$. Since $b$ is nondecreasing, $b(\theta)\ge b_0$ by \eqref{eq:decay-region}. Testing the chemical equation by $v$ yields
	 \begin{equation}\label{eq:v-diss}
	 \frac12\frac{d}{dt}\|v\|^2+d_c\|\nabla v\|^2
	 +\frac{b_0}{2}\|v\|^2\le\frac{M^2}{2b_0}\|u\|^2.
	 \end{equation}
	 For $s\in\mathbb R$,
	 \[
	 F''(s)=\frac12\bigl(1-6s+6s^2\bigr)=3\Bigl(s-\frac12\Bigr)^2-\frac14\ge-\frac14.
	 \]
	 The mean value theorem therefore gives $(F'(\bar\phi+w)-F'(\bar\phi))w\ge-\frac14|w|^2$.
	 Testing the phase equation by $w$ and applying Young's inequality to $\lambda\int_\Omega vw\,dx$ gives
	 \begin{equation}\label{eq:w-diss}
	 \frac{\tau_\phi}{2}\frac{d}{dt}\|w\|^2+\varepsilon^2\|\nabla w\|^2
	 +\frac{2\lambda-1}{4}\|w\|^2\le\frac\lambda2\|v\|^2.
	 \end{equation}
	 By \eqref{eq:mean-z}, $a\int_\Omega u=A\int_\Omega v+B\int_\Omega w$. Hence the Neumann--Poincar\'e inequality yields
	 \begin{equation}\label{eq:mean-est}
	 \|u\|^2\le C_P^2\|\nabla u\|^2
	 +\frac{2A^2}{a^2}\|v\|^2+\frac{2B^2}{a^2}\|w\|^2.
	 \end{equation}
	 Define the nonnegative quadratic energy by
	 \[
	 \mathscr E(t)=\frac12\|z(t)\|^2+\frac\mu2\|v(t)\|^2
	 +\frac{\nu\tau_\phi}{2}\|w(t)\|^2.
	 \]
	 Add \eqref{eq:z-diss}, $\mu$ times \eqref{eq:v-diss} and $\nu$ times \eqref{eq:w-diss}, and substitute \eqref{eq:mean-est}. Assumptions \eqref{eq:smallcoupling} give
	 \begin{equation}\label{eq:full-diss}
	 \begin{split}
	 \frac{d}{dt}\mathscr E
	 &+\frac{ak_*}{4}\|\nabla u\|^2
	 +\frac{\mu d_c}{2}\|\nabla v\|^2
	 +\frac{\nu\varepsilon^2}{2}\|\nabla w\|^2\\
	 &+\frac{\mu b_0}{4}\|v\|^2
	 +\frac{\nu(2\lambda-1)}{8}\|w\|^2\le0.
	 \end{split}
	 \end{equation}
	 Since $z=au-Av-Bw$, the norm of $(z,v,w)$ is equivalent to that of $(u,v,w)$, with constants depending on the fixed parameters. Furthermore, \eqref{eq:mean-est} gives
	 \[
	 \mathscr E(t)\le C_1\bigl(\|\nabla u\|^2+\|v\|^2+\|w\|^2\bigr).
	 \]
	 Thus \eqref{eq:full-diss} implies $\frac{d}{dt}\mathscr E+\varsigma\mathscr E\le0$ for some $\varsigma>0$. Integration and equivalence of norms prove \eqref{eq:conditional_decay}.
	 Finally, to see that the weight conditions are feasible for small $\alpha$, first choose $\mu>0$ sufficiently small to meet their first inequality. Next choose $\nu>0$ so that $\nu\lambda/2<\mu b_0/4$. At $A=B=0$ all the remaining inequalities are then strict, except for the trivial zero terms. Continuity guarantees them for sufficiently small $A=\alpha L_c$ and $B=\alpha L_\phi$.
	\end{proof}

	\begin{remark}
	 When \eqref{eq:source} holds with $H_{\mathrm{ext}}=0$, Theorem~\ref{thm:theta_pos} gives \eqref{eq:decay-region} on $[0,T]$ with $\theta_-=\min\{\bar\theta,\,\theta_*-C_0T/(\rho c_p)\}$ for every $T<T_0$. If $C_0=0$, this holds for every $T$ with $\theta_-=\min\{\bar\theta,\theta_*\}$.
	\end{remark}

	The condition $\lambda>1/2$ in Theorem~\ref{thm:conditional_decay} comes from Young's inequality in \eqref{eq:w-diss}. The modal analysis of Section~\ref{sec:local} requires only $\lambda>1/4$, which makes $s\mapsto F'(s)+\lambda s$ strictly increasing on $\mathbb R$.

	\section{Local well-posedness and stability near the equilibrium branch}
	\label{sec:local}

	Theorem~\ref{thm:conditional_decay} assumes a positive lower bound on the temperature. We now show that this bound holds for small perturbations of a homogeneous equilibrium, using the quasilinear structure in enthalpy variables.

	Fix $p>N+2$ and put
	\[
	 \begin{aligned}
	 X_0&=L^p(\Omega)^3,\\
	 X_1&=\bigl\{U\in W^{2,p}(\Omega)^3:\ \partial_{\mathbf n}U=0\ \text{on }\partial\Omega\bigr\},\\
	 X_\gamma&=(X_0,X_1)_{1-1/p,p}.
	 \end{aligned}
	\]
	Thus $X_\gamma$ is the Neumann subspace of $W^{2-2/p,p}(\Omega)^3$.
	In particular, $X_\gamma\hookrightarrow C^1(\overline\Omega)^3$.

	\begin{theorem}[Local well-posedness]
	\label{thm:localwp}
	 Assume \textup{\textbf{(H1)}}. Suppose in addition that
	 $\partial\Omega$ is of class $C^{2+\delta}$ for some $\delta>0$
	 and that $k\in C^3(\mathbb R)$ has bounded derivatives through order three.
	 Let $H_{\mathrm{ext}}=0$ and $(\theta_0,c_0,\phi_0)\in X_\gamma$.
	 Then the system with the extended Arrhenius rates possesses a unique maximal solution
	 on $[0,T_*)$, $0<T_*\le+\infty$, satisfying, for every $T<T_*$,
	 \begin{equation}\label{eq:localmr}
	 (\theta,c,\phi)\in H^{1,p}(0,T;X_0)\cap L^p(0,T;X_1)
	 \cap C([0,T];X_\gamma).
	 \end{equation}
	 Here $H^{1,p}(0,T;X_0)=W^{1,p}(0,T;X_0)$. The solution depends continuously on its initial data.
	\end{theorem}

	\begin{proof}
	 Write $a=\rho c_p$, $A=\alpha L_c$, $B=\alpha L_\phi$ and introduce
	 $h=a\theta-Ac-B\phi$. Put $\Theta(h,c,\phi)=a^{-1}(h+Ac+B\phi)$.
	 The transformed equations are
	 \begin{equation}\label{eq:parabolic-form}
	 \left\{\begin{aligned}
	 \partial_t h&=\nabla\cdot\bigl(k(\phi)\nabla\Theta(h,c,\phi)\bigr),\\
	 \partial_t c&=d_c\Delta c+\mathcal R(\Theta(h,c,\phi),c),\\
	 \partial_t\phi&=\frac{\varepsilon^2}{\tau_\phi}\Delta\phi
	 +\frac{\lambda(c-\phi)-F'(\phi)}{\tau_\phi},
	 \end{aligned}\right.
	 \end{equation}
	 where $\mathcal R$ is the reaction rate defined after \eqref{eq:CI}.
	 The Neumann conditions for the original three variables are equivalent to
	 $\partial_{\mathbf n}h=\partial_{\mathbf n}c=\partial_{\mathbf n}\phi=0$.
	 The coefficient matrix of the second-order principal part, in the order $(h,c,\phi)$, is
	 \[
	 \mathbb D(\phi)=
	 \begin{pmatrix}
	 k(\phi)/a&A k(\phi)/a&B k(\phi)/a\\
	 0&d_c&0\\0&0&\varepsilon^2/\tau_\phi
	 \end{pmatrix}.
	 \]
	 Its eigenvalues are $k(\phi)/a$, $d_c$ and $\varepsilon^2/\tau_\phi$,
	 all uniformly positive by \textbf{(H1)}.
	 For a frozen state $U_\sharp=(h_\sharp,c_\sharp,\phi_\sharp)\in X_\gamma$, set $\mathcal K_\sharp=a^{-1}\nabla\cdot(k(\phi_\sharp)\nabla\cdot)$ on $W^{2,p}_{\rm N}(\Omega)=\{u\in W^{2,p}(\Omega):\partial_{\mathbf n}u=0\text{ on }\partial\Omega\}$. The scalar Neumann operators $\mathcal K_\sharp$, $d_c\Delta$ and $(\varepsilon^2/\tau_\phi)\Delta$ have maximal $L^p$-regularity~\cite{DenkHieberPruss2003}. Solving successively for $\phi,c,h$ gives maximal regularity of the frozen triangular system: the off-diagonal sources $A\mathcal K_\sharp c+B\mathcal K_\sharp\phi$ lie in $L^p(0,T;L^p)$ by the scalar estimates. Writing the system as $U'+\mathcal A(U)U=\mathcal F(U)$ on the fixed domain $X_1$, the embedding $X_\gamma\hookrightarrow C^1(\overline\Omega)^3$, the regularity of $k$ and Lemma~\ref{lem:rates} ensure that $(\mathcal A,\mathcal F):X_\gamma\to\mathcal L(X_1,X_0)\times X_0$ is $C^1$. The quasilinear existence theorem~\cite{PrussSimonettZacher2009,PrussSimonett2016} gives \eqref{eq:localmr}, uniqueness and continuous dependence.
The transformation $(\theta,c,\phi)\leftrightarrow(h,c,\phi)$ is linear and invertible
	 and preserves these spaces, completing the proof.
	\end{proof}

	\begin{remark}
	 The smooth extension of the rates permits the local existence result
	 without assuming in advance that the evolving temperature remains positive.
	 Strict positivity and the bounds $0<c,\phi<1$ will follow from the
	 local stability theorem below, so the extension agrees with the
	 physical Arrhenius laws along these solutions.
	\end{remark}

	Fix a homogeneous equilibrium $(\theta_e,c_e,\phi_e)$ with $\theta_e>0$, and set
	$h_e=a\theta_e-Ac_e-B\phi_e$, $k_e=k(\phi_e)$, $b_e=b(\theta_e)$ and
	$r_e=\partial_\theta\mathcal R(\theta_e,c_e)$. The linearisation $\mathcal L_e$ of
	\eqref{eq:parabolic-form} at $(h_e,c_e,\phi_e)$ acts on perturbations $(z,v,w)$ through the system
	\begin{equation}\label{eq:linearised-matrix}
	\left\{\begin{aligned}
	\partial_tz&=\frac{k_e}{a}\Delta(z+Av+Bw),\\
	\partial_tv&=d_c\Delta v-b_ev+\frac{r_e}{a}(z+Av+Bw),\\
	\tau_\phi\partial_tw&=\varepsilon^2\Delta w
	-(F''(\phi_e)+\lambda)w+\lambda v,
	\end{aligned}\right.
	\end{equation}
	with homogeneous Neumann boundary conditions. The first component of $\mathcal L_eU$ is a Neumann Laplacian and has zero mean. Hence the subspace defined by
	\[
	X_0^0=\Bigl\{(z,v,w)\in X_0:\int_\Omega z\,dx=0\Bigr\}
	\]
	is invariant and contains the range of $\mathcal L_e$.

\subsection{Exact spectral threshold and stability}
\label{sec:exact-alpha-threshold}

The linearisation permits an exact stability criterion in terms of the thermal coupling parameter $\alpha$.
We retain the notation $a=\rho c_p$ and write
\[
 b=b_e,\quad r=r_e,\quad
 \delta_e=F''(\phi_e)+\lambda,\quad k_e=k(\phi_e),\quad
 \ell=\lambda/\tau_\phi.
\]
Here $(\theta_e,c_e,\phi_e)$ is a fixed homogeneous equilibrium with
$\theta_e>0$ and $\lambda>1/4$, so $b,\delta_e,k_e,\ell>0$.
Direct differentiation of $c_*$ gives
\begin{equation}\label{eq:alpha-activation}
 c_*'(\theta_e)=\frac rb
 =\frac{\beta\gamma K_d(\theta_e)K_r(\theta_e)}
 {R\theta_e^2 b^2}(E_r-E_d).
\end{equation}
Thus the sign of $r$ is that of $E_r-E_d$. In this model, $E_r<E_d$ corresponds to a decreasing equilibrium coloured fraction as temperature increases, whereas $E_r>E_d$ corresponds to an increasing fraction, consistent at the qualitative level with the reverse colour response reported in~\cite{DindaEtal2025}. The stability classification distinguishes these two equilibrium responses.
For a closed operator $L$ with spectrum $\sigma(L)$, we denote by $\operatorname{Re} s$ the real part of $s\in\mathbb C$, and we call
the spectral abscissa of $L$ the quantity defined by
\[
\sup\{\operatorname{Re} s:\ s\in\sigma(L)\}.
\]
Let $0=\Lambda_0<\Lambda_1\le\Lambda_2\le\cdots$
denote the Neumann eigenvalues of $-\Delta$, including multiplicities, and set
\[
 g_q=b+d_cq,\qquad
 H_q=\frac{\delta_e+\varepsilon^2q}{\tau_\phi},\qquad
 C_q=k_e q,\qquad\text{for } q=\Lambda_j.
\]
In particular $g_0=b$ and $H_0=\delta_e/\tau_\phi$.

\begin{theorem}[Exact thermal-coupling spectral criterion]
\label{thm:exact-alpha-threshold}
Assume \textup{\textbf{(H1)}}, $H_{\mathrm{ext}}=0$, $\lambda>1/4$ and the regularity hypotheses of Theorem~\ref{thm:localwp}.
Fix an equilibrium $(\theta_e,c_e,\phi_e)$ as above.
Let $\mathcal L_e$ be the linearisation in enthalpy variables.
\begin{enumerate}[label=\textup{(\roman*)}]
\item If $r\le0$, the spectrum of $\mathcal L_e$ restricted to
$X_0^0=\{(z,v,w)\in X_0:\int_\Omega z=0\}$ lies in a half-plane
$\{s\in\mathbb C:\operatorname{Re} s\le-\omega\}$, with $\omega>0$, for every fixed $\alpha\ge0$.
The same holds for $r>0$ when $L_c=L_\phi=0$.
\item If $r>0$ and $L_c+L_\phi>0$, put
\begin{equation}\label{eq:alpha-crit}
 \alpha_{\mathrm{crit}}
 =\frac{abH_0}{r(L_cH_0+L_\phi\ell)}
 =\frac{a}{c_*'(\theta_e)
 \left(L_c+\dfrac{\lambda L_\phi}{\delta_e}\right)}.
\end{equation}
For $0\le\alpha<\alpha_{\mathrm{crit}}$ the restricted linearisation
is exponentially stable.
At $\alpha=\alpha_{\mathrm{crit}}$ an additional zero eigenvalue
appears on $X_0^0$.
For $\alpha>\alpha_{\mathrm{crit}}$ the restricted linearisation
has a positive real eigenvalue.
\end{enumerate}
In particular, no nonconstant Neumann mode loses stability
before the homogeneous one.
\end{theorem}

\begin{proof}
The linearised system is \eqref{eq:linearised-matrix}.
For a Neumann mode with eigenvalue $q\ge0$, we seek solutions $e^{st}(z,v,w)\psi_q$, where $\psi_q$ is a Neumann eigenfunction with $-\Delta\psi_q=q\psi_q$. We replace $z$ by $u=(z+\alpha L_cv+\alpha L_\phi w)/a$. This change of unknowns is invertible. System \eqref{eq:linearised-matrix} becomes
\[
\left\{\begin{aligned}
(as+C_q)\,u-\alpha L_c s\,v-\alpha L_\phi s\,w&=0,\\
-r\,u+(s+g_q)\,v&=0,\\
-\ell\,v+(s+H_q)\,w&=0.
\end{aligned}\right.
\]
The first equation follows from $sz=-C_qu$ and $z=au-\alpha L_cv-\alpha L_\phi w$. The other two are the modal forms of the chemical and phase equations. Expanding the determinant of this system along its first row gives
\[
(as+C_q)(s+g_q)(s+H_q)-\alpha L_c s\,r(s+H_q)-\alpha L_\phi s\,r\ell .
\]
Hence $s$ is an eigenvalue in the mode $q$ if and only if it is a root of
\begin{equation}\label{eq:alpha-poly}
 P_q(s)=(as+C_q)(s+g_q)(s+H_q)
 -\alpha r s\bigl[L_c(s+H_q)+L_\phi\ell\bigr].
\end{equation}
Write $P_q(s)=as^3+B_2(q)s^2+B_1(q)s+B_0(q)$, where
\begin{equation}\label{eq:alpha-coeff}
\begin{aligned}
 B_2(q)&=a(g_q+H_q)+C_q-\alpha rL_c,\\
 B_1(q)&=ag_qH_q+C_q(g_q+H_q)
 -\alpha r(L_cH_q+L_\phi\ell),\\
 B_0(q)&=C_qg_qH_q.
\end{aligned}
\end{equation}
For $q>0$, the Routh--Hurwitz criterion requires
$B_2(q),B_1(q),B_0(q)>0$ and $B_2(q)B_1(q)>aB_0(q)$.
For $q=0$, the polynomial factors as
\begin{equation}\label{eq:alpha-zero}
P_0(s)=s\left[as^2+T_\alpha s+D_\alpha\right],
\quad\text{with}\quad\begin{aligned}
T_\alpha&=a(b+H_0)-\alpha rL_c,\\
D_\alpha&=abH_0-\alpha r(L_cH_0+L_\phi\ell).
\end{aligned}
\end{equation}
The factor $s$ comes from conservation of the mean enthalpy.
The two roots of the quadratic factor describe the homogeneous
perturbations on $X_0^0$.

First let $r\le0$. Define
\[
U_q=a(g_q+H_q)+C_q,\quad
V_q=ag_qH_q+C_q(g_q+H_q),\quad W_q=C_qg_qH_q.
\]
For $q>0$, $B_2(q)\ge U_q>0$ and $B_1(q)\ge V_q>0$.
Moreover, the following identity holds:
\[
U_qV_q-aW_q
=a^2(g_q+H_q)g_qH_q+aC_q(g_q+H_q)^2
+C_q^2(g_q+H_q)>0.
\]
The Routh--Hurwitz conditions follow.
For $q=0$, $T_\alpha\ge a(b+H_0)>0$ and
$D_\alpha\ge abH_0>0$.
If $L_c=L_\phi=0$, the polynomial is independent of $\alpha$.

Now assume $r>0$ and $L_c+L_\phi>0$.
For $0\le\alpha\le\alpha_{\mathrm{crit}}$ one has
\[
\alpha rL_c\le ab,\qquad
\alpha r(L_cH_q+L_\phi\ell)\le abH_q.
\]
This follows from $H_q\ge H_0$ and
$H_0(L_cH_q+L_\phi\ell)\le H_q(L_cH_0+L_\phi\ell)$.
Consequently, for $q>0$,
\[
 B_2(q)\ge aH_q+C_q>0,\qquad
 B_1(q)\ge ad_cqH_q+C_q(g_q+H_q)>0.
\]
Hence the following estimate holds:
\begin{align*}
 B_2(q)B_1(q)-aB_0(q)
 &\ge(aH_q+C_q)\bigl(ad_cqH_q+C_q(g_q+H_q)\bigr)
   -aC_qg_qH_q\\
 &=a^2d_cqH_q^2+aC_qH_q^2\\
 &\quad+aC_qd_cqH_q+C_q^2(g_q+H_q)>0.
\end{align*}
Thus every nonconstant mode remains stable, even at the threshold.
For the homogeneous mode, one has
\[
T_\alpha\ge aH_0>0\quad\text{for } 0\le\alpha\le\alpha_{\mathrm{crit}},
\qquad\text{and}\qquad
D_\alpha=abH_0\left(1-\frac{\alpha}{\alpha_{\mathrm{crit}}}\right).
\]
The quadratic is strictly stable before the threshold.
At the threshold it has a zero root.
Above the threshold $D_\alpha<0$, so one root is real and positive.
For a nonzero eigenvalue of the homogeneous block,
conservation gives $z=0$, so its eigenvector belongs to $X_0^0$.

It remains to justify a uniform spectral gap.
The modal matrices have the form $-q\mathbb D_e+\mathbb R_e$,
where $\mathbb D_e=\mathbb D(\phi_e)$ is the diffusion matrix of the proof of Theorem~\ref{thm:localwp} and $\mathbb R_e$ is the matrix of the zero-order terms of \eqref{eq:linearised-matrix}. The matrix $\mathbb D_e$ is upper triangular with strictly positive diagonal
entries $k_e/a$, $d_c$ and $\varepsilon^2/\tau_\phi$.
By continuity of the roots of the characteristic polynomial, the eigenvalues of
$-\mathbb D_e+q^{-1}\mathbb R_e$ converge, counting multiplicities, to the eigenvalues of $-\mathbb D_e$ as $q\to\infty$. Therefore the spectral abscissae of the modal matrices tend to $-\infty$.
Only finitely many modes remain, each strictly stable in the
stated parameter range.
The $L^p$ realisation has compact resolvent by maximal regularity and the compact embedding $X_1\hookrightarrow X_0$. Every $L^p$ eigenfunction lies in $L^2$ and its Neumann expansion identifies an eigenvalue of a modal matrix. Conversely, each modal eigenvector times its Neumann eigenfunction lies in $X_1$ by elliptic regularity. Hence the modal spectra exhaust the $L^p$ spectrum, also on $X_0^0$ where the constant block is restricted to $z=0$. The uniform gap follows.
\end{proof}

The threshold has a direct interpretation through the enthalpy
of the equilibrium branch. By Proposition~\ref{prop:equilibrium_branch},
\begin{equation}\label{eq:alpha-enthalpy-derivative}
\mathscr H'(\theta_e)
=a-\alpha c_*'(\theta_e)
\left(L_c+\frac{\lambda L_\phi}{\delta_e}\right)
=\frac{D_\alpha}{bH_0}.
\end{equation}
Thus $\mathscr H'(\theta_e)=0$ precisely at the spectral
threshold, where the equilibrium curve ceases to be transverse to the
fixed-enthalpy subspace. The nature of the transition is determined by
the higher derivatives of $\mathscr H$.

\begin{theorem}[Local stability, equilibrium selection and instability]
\label{thm:local-stability}
Assume \textup{\textbf{(H1)}}, $H_{\mathrm{ext}}=0$, $\lambda>1/4$ and the regularity hypotheses of Theorem~\ref{thm:localwp}.
Let $(\theta_e,c_e,\phi_e)$ be a homogeneous equilibrium with $\theta_e>0$.
\begin{enumerate}[label=\textup{(\roman*)}]
\item \emph{Stability.} Assume that $r\le0$, or that $r>0$ and $L_c=L_\phi=0$, or that $r>0$, $L_c+L_\phi>0$ and $\alpha<\alpha_{\mathrm{crit}}$.
There exist $\rho_0,C,\omega>0$ such that every Neumann-compatible initial state $(\theta_0,c_0,\phi_0)\in X_\gamma$ satisfying
\[
\|(\theta_0,c_0,\phi_0)-(\theta_e,c_e,\phi_e)\|_{X_\gamma}<\rho_0
\]
generates a unique global solution. There exists a unique homogeneous equilibrium $(\theta_\infty,c_\infty,\phi_\infty)$ near $(\theta_e,c_e,\phi_e)$ selected by
\begin{equation}\label{eq:select-ent}
a\theta_\infty-Ac_\infty-B\phi_\infty
=\frac1{|\Omega|}\int_\Omega(a\theta_0-Ac_0-B\phi_0)\,dx,
\end{equation}
and
\begin{equation}\label{eq:global-stability}
\begin{aligned}
&\|(\theta(t),c(t),\phi(t))-(\theta_\infty,c_\infty,\phi_\infty)\|_{X_\gamma}\\
&\qquad\le Ce^{-\omega t}
\|(\theta_0,c_0,\phi_0)-(\theta_e,c_e,\phi_e)\|_{X_\gamma},
\quad\text{for } t\ge0.
\end{aligned}
\end{equation}
Moreover, $\theta\ge\theta_e/2$ and $0<c,\phi<1$ on $\overline\Omega\times[0,\infty)$.
\item \emph{Instability.} If $r>0$, $L_c+L_\phi>0$ and $\alpha>\alpha_{\mathrm{crit}}$, the equilibrium is nonlinearly unstable, even on its fixed-enthalpy class.
\end{enumerate}
\end{theorem}

\begin{proof}
\emph{Interior equilibrium.} Both Arrhenius rates are positive at $\theta_e$, so $c_e=c_*(\theta_e)\in(0,1)$.
The map $s\mapsto F'(s)+\lambda s$ is strictly increasing on $\mathbb R$, takes the value $0<\lambda c_e$ at $s=0$ and the value $\lambda>\lambda c_e$ at $s=1$. Hence $\phi_e\in(0,1)$.

\emph{Branch.} Under the assumptions of (i), \eqref{eq:alpha-enthalpy-derivative} gives $\mathscr H'(\theta_e)>0$, since $D_\alpha>0$ in each of the three cases.
By Proposition~\ref{prop:equilibrium_branch}, the equilibria near $(\theta_e,c_e,\phi_e)$ form the $C^1$ curve $\mathcal E=\{(\mathscr H(\theta),c_*(\theta),\Phi(\theta))\}$ in the variables $(h,c,\phi)$. By Lemma~\ref{lem:stationary-completeness}, this curve exhausts all nearby stationary solutions. Its tangent vector at $\theta_e$ is given by
\[
\tau_e=\bigl(\mathscr H'(\theta_e),\,c_*'(\theta_e),\,\Phi'(\theta_e)\bigr)\neq0.
\]

\emph{Kernel and semi-simplicity.} Differentiating the equilibrium equation gives $\mathcal L_e\tau_e=0$. Since $\mathscr H'(\theta_e)>0$, one has $X_0=\mathrm{span}\{\tau_e\}\oplus X_0^0$. Theorem~\ref{thm:exact-alpha-threshold} gives a negative spectral bound for $\mathcal L_e|_{X_0^0}$, so this restriction is invertible. As the range of $\mathcal L_e$ lies in $X_0^0$, it follows that
\[
\ker\mathcal L_e=\mathrm{span}\{\tau_e\}=T_{(h_e,c_e,\phi_e)}\mathcal E,\qquad \operatorname{Ran}\mathcal L_e=X_0^0.
\]
Thus $0$ is semisimple and the remaining spectrum lies in $\{s\in\mathbb C:\operatorname{Re}s\le-\omega\}$ for some $\omega>0$.

\emph{Convergence and selection.} By Theorem~\ref{thm:localwp}, problem \eqref{eq:parabolic-form} has maximal $L^p$-regularity. The equilibrium is therefore normally stable, and the generalised principle of linearised stability \cite[Theorem~2.1]{PrussSimonettZacher2009} yields $\rho_0>0$ such that every solution starting in the $X_\gamma$-ball of radius $\rho_0$ around $(h_e,c_e,\phi_e)$ exists globally and converges at an exponential rate in $X_\gamma$ to some equilibrium $(h_\infty,c_\infty,\phi_\infty)\in\mathcal E$. More precisely, the estimate in \cite[p.~3910, Eq.~(2.27)]{PrussSimonettZacher2009} bounds the distance to the limit by $\widetilde Me^{-\omega t}\|P^sV_0-\Psi(P^cV_0)\|_{X_\gamma}$, where $\widetilde M>0$ is a constant, $V_0$ is the initial perturbation, $P^c,P^s$ are the spectral projections and $\Psi$ parametrises the equilibrium manifold. Since $\Psi(0)=0$ and $D\Psi(0)=0$, the last norm is bounded by a constant times $\|V_0\|_{X_\gamma}$ on a sufficiently small ball. This gives constants $C,\omega>0$, depending only on the reference equilibrium, such that the distance to the limit is bounded by $Ce^{-\omega t}$ times the initial distance to $(h_e,c_e,\phi_e)$. The linear change of variables $(h,c,\phi)\leftrightarrow(\theta,c,\phi)$ gives \eqref{eq:global-stability}. By Proposition~\ref{prop:enthalpy}, the mean of $h$ is conserved. This is \eqref{eq:select-ent}, and the local injectivity of $\mathscr H$ identifies the limit.

\emph{Physical bounds.} The selected equilibrium satisfies $|\theta_\infty-\theta_e|+|c_\infty-c_e|+|\phi_\infty-\phi_e|\le C'\rho_0$, by the continuity of $\mathscr H^{-1}$. Since $X_\gamma\hookrightarrow C^1(\overline\Omega)^3$, estimate \eqref{eq:global-stability} and the interior bounds $\theta_e>0$, $c_e,\phi_e\in(0,1)$ give $\theta\ge\theta_e/2$ and $0<c,\phi<1$ for all $t\ge0$, after reducing $\rho_0$.

\emph{Instability.} Assume the hypotheses of (ii). Spatially homogeneous solutions with a prescribed enthalpy value $h_0$ satisfy
\[
\theta=\frac{h_0+\alpha L_cc+\alpha L_\phi\phi}{a}
\]
and therefore solve the autonomous system
\begin{equation}\label{eq:homogeneous-instability-ode}
\left\{\begin{aligned}
\dot c&=\mathcal R\!\left(\frac{h_0+\alpha L_cc+\alpha L_\phi\phi}{a},c\right),\\
\dot\phi&=\frac{\lambda(c-\phi)-F'(\phi)}{\tau_\phi}.
\end{aligned}\right.
\end{equation}
At the equilibrium corresponding to $h_0=h_e$, its Jacobian is given by
\[
J_0(\alpha)=\begin{pmatrix}
-b+\alpha rL_c/a&\alpha rL_\phi/a\\[2pt]
\ell&-H_0
\end{pmatrix},\qquad\text{with}\qquad
\det J_0(\alpha)=\frac{D_\alpha}{a}<0.
\]
The real matrix $J_0(\alpha)$ has a positive eigenvalue, so the equilibrium of \eqref{eq:homogeneous-instability-ode} is nonlinearly unstable. The class of homogeneous solutions with fixed enthalpy is invariant under the flow of \eqref{eq:sys}, hence the equilibrium of the partial differential system is nonlinearly unstable.
\end{proof}

\begin{remark}
Theorem~\ref{thm:local-stability} does not require the source condition \eqref{eq:source}. The positivity of the temperature follows from the proximity to the equilibrium.
\end{remark}

\begin{proposition}[Critical slowing down of the homogeneous mode]
\label{prop:critical-slowing}
Under the assumptions of Theorem~\ref{thm:exact-alpha-threshold},
assume $r>0$ and $L_c+L_\phi>0$, and keep the homogeneous equilibrium
fixed as $\alpha$ varies. Put
\[
 T_{\mathrm{crit}}=T_{\alpha_{\mathrm{crit}}}>0,\qquad
 \eta=1-\frac{\alpha}{\alpha_{\mathrm{crit}}}.
\]
The homogeneous eigenvalue on the fixed-enthalpy subspace which
vanishes at the threshold satisfies
\begin{equation}\label{eq:critical-slowing}
 s_{\mathrm{slow}}(\alpha)
 =-\frac{abH_0}{T_{\mathrm{crit}}}\eta+O(\eta^2),
 \quad\text{as }\eta\to0.
\end{equation}
For $\alpha<\alpha_{\mathrm{crit}}$ sufficiently close to the
threshold, $s_{\mathrm{slow}}(\alpha)$ is the spectral abscissa of
$\mathcal L_e|_{X_0^0}$. In particular, the exponential decay rate tends to zero
linearly as $\alpha\uparrow\alpha_{\mathrm{crit}}$.
\end{proposition}

\begin{proof}
The two eigenvalues of the homogeneous block restricted to the fixed-enthalpy subspace are the roots of
$a s^2+T_\alpha s+D_\alpha=0$, where
$D_\alpha=abH_0\eta$ and
$T_\alpha=T_{\mathrm{crit}}+O(\eta)$. At $(s,\eta)=(0,0)$ the derivative of this
polynomial with respect to $s$ is $T_{\mathrm{crit}}>0$. The implicit-function
theorem, or the quadratic formula, gives the unique root near zero
and the expansion \eqref{eq:critical-slowing}. The other homogeneous
root tends to $-T_{\mathrm{crit}}/a<0$. At $\alpha=\alpha_{\mathrm{crit}}$ all
nonconstant modes are strictly stable by
Theorem~\ref{thm:exact-alpha-threshold}. For large Neumann
eigenvalues $q$, the modal matrices divided by $q$ converge,
uniformly for $\alpha$ in a compact neighbourhood of the threshold,
to $-\mathbb D_e(\alpha)$, whose eigenvalues are the three negative
diffusion constants independent of $\alpha$. The remaining finitely
many nonconstant modes have a common negative spectral bound by
continuity in $\alpha$. Hence the root
$s_{\mathrm{slow}}(\alpha)$ determines the spectral abscissa
for $\alpha<\alpha_{\mathrm{crit}}$ sufficiently close to the
threshold.
\end{proof}

\begin{remark}
The threshold \eqref{eq:alpha-crit} is computed at a fixed equilibrium
$(\theta_e,c_e,\phi_e)$ as $\alpha$ varies. If instead the total enthalpy
is fixed while $\alpha$ varies, the equilibrium itself may change.
The transition is then determined implicitly by
$\mathscr H'(\theta_e(\alpha))=0$.
\end{remark}

\subsection{Bifurcation at the adiabatic threshold}
\label{sec:bifurcation}

At $\alpha=\alpha_{\mathrm{crit}}$ the derivative $\mathscr H'(\theta_e)$ of the enthalpy map vanishes. A vanishing derivative alone does not imply that the enthalpy fails to select the equilibrium uniquely. Under the nondegeneracy condition $\kappa_e\neq0$ introduced below, the enthalpy map has a fold at $\theta_e$, and along the parameter-dependent classes $h_\alpha$ two branches of equilibria intersect. We now describe the equilibria and the dynamics near this point. The homogeneous equilibria $(\theta,c_*(\theta),\Phi(\theta))$ do not depend on $\alpha$. Only the enthalpy of the branch depends on $\alpha$, through the map defined by
\begin{equation}\label{eq:H-alpha}
\mathscr H_\alpha(\theta)=a\theta-\alpha\bigl(L_cc_*(\theta)+L_\phi\Phi(\theta)\bigr).
\end{equation}
Differentiating $F'(\Phi)+\lambda\Phi=\lambda c_*$ twice gives $\Phi'=\lambda c_*'/\delta$ and
\[
\Phi''=\frac{\lambda c_*''-F'''(\Phi)\,(\Phi')^2}{\delta},\quad\text{with}\quad \delta=F''(\Phi)+\lambda\quad\text{and}\quad F'''(s)=6s-3 .
\]
We set $\eta=1-\alpha/\alpha_{\mathrm{crit}}$, as in Proposition~\ref{prop:critical-slowing}, and
\begin{equation}\label{eq:kappa-e}
\kappa_e=L_cc_*''(\theta_e)+L_\phi\Phi''(\theta_e).
\end{equation}
By \eqref{eq:alpha-enthalpy-derivative}, $\mathscr H_\alpha'(\theta_e)=a\eta$, and $\mathscr H_{\alpha_{\mathrm{crit}}}''(\theta_e)=-\alpha_{\mathrm{crit}}\kappa_e$.

\begin{theorem}[Bifurcation at the adiabatic threshold]
\label{thm:bifurcation}
Assume \textup{\textbf{(H1)}}, $H_{\mathrm{ext}}=0$, $\lambda>1/4$ and the regularity hypotheses of Theorem~\ref{thm:localwp}. Fix a homogeneous equilibrium $(\theta_e,c_e,\phi_e)$ with $\theta_e>0$, $r>0$ and $L_c+L_\phi>0$, and assume $\kappa_e\neq0$. Let $h_\alpha=\mathscr H_\alpha(\theta_e)$ denote the enthalpy class of $(\theta_e,c_e,\phi_e)$. There exist $\eta_0>0$ and a neighbourhood $\mathcal U$ of $(\theta_e,c_e,\phi_e)$ in $X_\gamma$ such that the following statements hold.
\begin{enumerate}[label=\textup{(\roman*)}]
\item \emph{Transcritical bifurcation.} For $|\eta|<\eta_0$, the equilibria in $\mathcal U$ with enthalpy $h_\alpha$ are $(\theta_e,c_e,\phi_e)$ and $(\theta_2,c_*(\theta_2),\Phi(\theta_2))$, where $\theta_2=\theta_2(\alpha)$ depends smoothly on $\alpha$. They are distinct for $0<|\eta|<\eta_0$ and coincide for $\eta=0$. Moreover, the following expansion holds:
\begin{equation}\label{eq:theta2}
\theta_2-\theta_e=\frac{2a}{\alpha_{\mathrm{crit}}\kappa_e}\,\eta+O(\eta^2).
\end{equation}
\item \emph{Exchange of stability.} Let $0<|\eta|<\eta_0$. For $\eta>0$ the equilibrium $(\theta_e,c_e,\phi_e)$ is normally stable and $(\theta_2,c_*(\theta_2),\Phi(\theta_2))$ is nonlinearly unstable. For $\eta<0$ the roles are exchanged. Solutions starting near the normally stable equilibrium exist globally and converge exponentially to the equilibrium selected by their enthalpy, as in Theorem~\ref{thm:local-stability}\textup{(i)}. In particular they converge to the normally stable equilibrium itself when their enthalpy is $h_\alpha$.
\item \emph{Critical dynamics.} For $\eta=0$ the equilibrium $(\theta_e,c_e,\phi_e)$ is nonlinearly unstable in $X_\gamma$. In the class of homogeneous solutions with enthalpy $h_e$, it possesses a one-dimensional local centre manifold, on which the reduced equation is
\begin{equation}\label{eq:centre-dynamics}
\dot y=\varrho\,y^2+O(|y|^3),\quad\text{with}\quad \varrho=\frac{\alpha_{\mathrm{crit}}bH_0\kappa_e}{2T_{\mathrm{crit}}}\neq0,
\end{equation}
where $y$ is a local coordinate on the centre manifold normalised by $\theta-\theta_e=y+O(y^2)$.
Solutions starting on this manifold with $\varrho y(0)<0$ satisfy $y(t)=-1/(\varrho t)+O(t^{-2}\log t)$, and solutions with $\varrho y(0)>0$ leave a fixed neighbourhood.
\item \emph{Fold of the enthalpy map.} For fixed $\alpha$ with $|\eta|<\eta_0$, the map $\mathscr H_\alpha$ has a unique critical point $\theta_f(\alpha)$ near $\theta_e$. It is nondegenerate and $\theta_f=\theta_e$ for $\eta=0$. An enthalpy class $h$ close to $\mathscr H_\alpha(\theta_f)$ contains two equilibria near $\theta_e$ if $(\mathscr H_\alpha(\theta_f)-h)\kappa_e>0$, one if $h=\mathscr H_\alpha(\theta_f)$, and none if $(\mathscr H_\alpha(\theta_f)-h)\kappa_e<0$.
\end{enumerate}
\end{theorem}

\begin{proof}
By Lemma~\ref{lem:stationary-completeness}, every equilibrium in $\mathcal U$ is homogeneous and of the form $(\theta,c_*(\theta),\Phi(\theta))$, so all statements on equilibria reduce to the scalar equation $\mathscr H_\alpha(\theta)=h$.

\emph{(i)} Define $g(\theta,\alpha)=(\mathscr H_\alpha(\theta)-\mathscr H_\alpha(\theta_e))/(\theta-\theta_e)$ for $\theta\neq\theta_e$ and $g(\theta_e,\alpha)=\mathscr H_\alpha'(\theta_e)$. The function $g$ is smooth, with $g(\theta_e,\alpha)=a\eta$ and $\partial_\theta g(\theta_e,\alpha_{\mathrm{crit}})=\tfrac12\mathscr H''_{\alpha_{\mathrm{crit}}}(\theta_e)=-\tfrac12\alpha_{\mathrm{crit}}\kappa_e\neq0$. The implicit function theorem gives a unique smooth solution $\theta_2(\alpha)$ of $g=0$ near $(\theta_e,\alpha_{\mathrm{crit}})$, and
\[
\theta_2-\theta_e=-\frac{g(\theta_e,\alpha)}{\partial_\theta g(\theta_e,\alpha_{\mathrm{crit}})}+O(\eta^2)
=\frac{2a\eta}{\alpha_{\mathrm{crit}}\kappa_e}+O(\eta^2).
\]
The solutions of $\mathscr H_\alpha(\theta)=h_\alpha$ near $\theta_e$ are $\theta_e$ and the zeros of $g$, which proves (i).

\emph{(ii)} Identity \eqref{eq:alpha-enthalpy-derivative} holds at every homogeneous equilibrium, with $b,r,H_0,\ell$ evaluated there. At $\theta_2$, a Taylor expansion and \eqref{eq:theta2} give
\[
\begin{aligned}
\mathscr H_\alpha'(\theta_2)&=\mathscr H_\alpha'(\theta_e)+\mathscr H_\alpha''(\theta_e)(\theta_2-\theta_e)+O(\eta^2)\\
&=a\eta-2a\eta+O(\eta^2)=-a\eta+O(\eta^2).
\end{aligned}
\]
Hence $D_\alpha$ has the sign of $\eta$ at $\theta_e$ and the sign of $-\eta$ at $\theta_2$. Moreover, $T_\alpha$ remains positive near $(\theta_e,\alpha_{\mathrm{crit}})$. By the proof of Proposition~\ref{prop:critical-slowing}, the nonconstant modes at $(\theta_e,\alpha_{\mathrm{crit}})$ have a common negative spectral bound, and this bound persists at $\theta_2$ for small $|\eta|$, since the modal coefficients depend continuously on the equilibrium. The equilibrium with $D_\alpha>0$ is therefore normally stable, and the proof of Theorem~\ref{thm:local-stability}\textup{(i)} applies to it without change. At the equilibrium with $D_\alpha<0$, the homogeneous block has a positive eigenvalue, and the reduction to the homogeneous system \eqref{eq:homogeneous-instability-ode} used in Theorem~\ref{thm:local-stability}\textup{(ii)} gives nonlinear instability.

\emph{(iii)} For $\eta=0$, system \eqref{eq:homogeneous-instability-ode} with $h_0=h_e$ is a planar system for $(c,\phi)$ whose Jacobian has the eigenvalues $0$ and $-T_{\mathrm{crit}}/a<0$, by the proof of Proposition~\ref{prop:critical-slowing}. We consider the extended system obtained by adding the equation $\dot\eta=0$ and by keeping $h_0=h_\alpha$. The vector field is of class $C^\infty$, so the centre manifold theorem~\cite{Carr1981} gives, for every $k$, a two-dimensional local centre manifold of class $C^k$ on a neighbourhood depending on $k$. We use it with $k=3$, which justifies the remainder $O(|y|^3)$ below. It is parametrised by $(y,\eta)$, and the reduced equation is $\dot y=y\,m(y,\eta)$, because $y=0$ is the equilibrium $\theta_e$ for every $\eta$. The function $m$ is $C^1$, $m(0,\eta)$ is the slow eigenvalue at $\theta_e$, and Proposition~\ref{prop:critical-slowing} gives $\partial_\eta m(0,0)=-abH_0/T_{\mathrm{crit}}\neq0$.

The zero eigenvector $(v,w)$ of the homogeneous block satisfies $w=\ell v/H_0$ with $v\neq0$, and the corresponding variation of temperature is $\alpha(L_cv+L_\phi w)/a\neq0$, since $L_c+L_\phi>0$. Hence $y$ and $\theta-\theta_e$ are related by a local $C^1$ diffeomorphism along the centre manifold, and the nonzero equilibria of the reduced equation are the points corresponding to $\theta_2(\alpha)$. By the implicit function theorem, the zero set of $m$ near $(0,0)$ is a curve $\eta=\zeta(y)$ with $\zeta'(0)=-\partial_ym(0,0)/\partial_\eta m(0,0)$. By (i), this curve is $y\mapsto\eta$ along $\theta_2$, whose derivative at $y=0$ is finite and nonzero by \eqref{eq:theta2}. Therefore $\varrho=\partial_ym(0,0)\neq0$, and for $\eta=0$ the reduced equation is \eqref{eq:centre-dynamics}. With the coordinate normalised by $\theta-\theta_e=y+O(y^2)$, \eqref{eq:theta2} gives $y_2(\eta)=2a\eta/(\alpha_{\mathrm{crit}}\kappa_e)+O(\eta^2)$. Since $m(y_2(\eta),\eta)=0$, $m(0,\eta)=s_{\mathrm{slow}}$ and $\partial_\eta m(0,0)=-abH_0/T_{\mathrm{crit}}$, we obtain
\[
\varrho=-\partial_\eta m(0,0)\lim_{\eta\to0}\frac{\eta}{y_2(\eta)}=\frac{abH_0}{T_{\mathrm{crit}}}\cdot\frac{\alpha_{\mathrm{crit}}\kappa_e}{2a}=\frac{\alpha_{\mathrm{crit}}bH_0\kappa_e}{2T_{\mathrm{crit}}} .
\]
If $\varrho y(0)>0$, then $\dot y$ has the sign of $y$ near $0$ and $|y|$ increases until $y$ leaves a fixed neighbourhood. If $\varrho y(0)<0$ and $|y(0)|$ is small, then $y$ keeps its sign and tends to $0$. Put $z=-1/(\varrho y)>0$. Then $z$ satisfies
\[
\dot z=\frac{\dot y}{\varrho y^2}=1+O(|y|)=1+O\Bigl(\frac1z\Bigr),
\]
so $z(t)=t+O(\log t)$ and
\[
y(t)=-\frac{1}{\varrho t}+O\Bigl(\frac{\log t}{t^2}\Bigr),\quad\text{as } t\to+\infty.
\]
Since the second eigenvalue is negative, the instability on the centre manifold implies the instability of the equilibrium of the planar system~\cite{Carr1981}. The class of homogeneous solutions with fixed enthalpy is invariant under the flow of \eqref{eq:sys}, so the equilibrium is nonlinearly unstable in $X_\gamma$.

\emph{(iv)} The equation $\mathscr H_\alpha'(\theta)=0$ has the solution $\theta_e$ for $\eta=0$, and $\mathscr H''_{\alpha_{\mathrm{crit}}}(\theta_e)\neq0$. The implicit function theorem gives a unique smooth $\theta_f(\alpha)$ near $\theta_e$. Set $\kappa(\theta)=L_cc_*''(\theta)+L_\phi\Phi''(\theta)$. For each fixed $\alpha$ near $\alpha_{\mathrm{crit}}$, Taylor's formula at $\theta_f=\theta_f(\alpha)$ gives
\[
\mathscr H_\alpha(\theta)-h=\mathscr H_\alpha(\theta_f)-h-\frac{\alpha}{2}\kappa(\theta_f)(\theta-\theta_f)^2+O(|\theta-\theta_f|^3).
\]
By continuity, $\kappa(\theta_f)$ is nonzero and has the sign of $\kappa_e$. After reducing the neighbourhood, $\mathscr H_\alpha$ is strictly monotone on either side of $\theta_f$, with opposite monotonicities. Its local maximum (respectively minimum) is attained at $\theta_f$ when $\kappa_e>0$ (respectively $\kappa_e<0$). For $h$ sufficiently close to the extremal value, the equation has exactly two local roots if $(\mathscr H_\alpha(\theta_f)-h)\kappa_e>0$, one double root at equality and no local root in the opposite case.
\end{proof}

\begin{remark}
When the enthalpy ceases to select the equilibrium uniquely, two equilibria of the same enthalpy class meet and exchange their stability. At the threshold itself the equilibrium is attracting on one side of the centre manifold and repelling on the other. On the attracting side the relaxation is algebraic rather than exponential. For a fixed coupling near $\alpha_{\mathrm{crit}}$, the enthalpy classes near the fold value contain two equilibria, one stable and one unstable, or none. The sign of $\kappa_e$, which is computable from the Arrhenius laws and the potential $F$, determines the side of the fold. The condition $\kappa_e\neq0$ does not follow from $r>0$. Writing $c_*=1/(1+A_0e^{D_E/\theta})$ with $A_0=\beta A_d/(\gamma A_r)$ and $D_E=(E_r-E_d)/R$, one finds $c_*''=c_*(1-c_*)\bigl[(1-2c_*)D_E^2/\theta^4-2D_E/\theta^3\bigr]$. For $L_\phi=0$, $D_E=4\theta_e$ and $c_*(\theta_e)=1/4$, this gives $\kappa_e=0$ while $r>0$. In the degenerate case $\kappa_e=0$, the transition is governed by terms of higher order.
\end{remark}

\begin{corollary}[The symmetric equilibrium]
\label{cor:symmetric}
Under the assumptions of Theorem~\ref{thm:bifurcation} without the condition $\kappa_e\neq0$, let $c_e=\phi_e=1/2$. Then the following formulas hold:
\[
\kappa_e=-\frac{2a}{\alpha_{\mathrm{crit}}\theta_e}<0,\qquad
\theta_2-\theta_e=-\theta_e\,\eta+O(\eta^2),\qquad
\varrho=-\frac{abH_0}{T_{\mathrm{crit}}\theta_e}<0 .
\]
In particular the nondegeneracy condition holds, and at $\alpha=\alpha_{\mathrm{crit}}$ the critical equilibrium attracts the perturbations of the centre manifold that raise the temperature and repels those that lower it.
\end{corollary}

\begin{proof}
Write $\beta K_d/(\gamma K_r)=e^{-\upsilon}$ with $\upsilon(\theta)=-\ln A_0-D_E/\theta$, where $A_0=\beta A_d/(\gamma A_r)$ and $D_E=(E_r-E_d)/R>0$ because $r>0$. Then $c_*=1/(1+e^{-\upsilon})$, so that $c_*'=c_*(1-c_*)\upsilon'$ and $c_*''=c_*(1-c_*)(1-2c_*)(\upsilon')^2+c_*(1-c_*)\upsilon''$. For $c_*=1/2$ and $\upsilon''=-2\upsilon'/\theta$, this gives $c_*''(\theta_e)=-2c_*'(\theta_e)/\theta_e$. Since $F'(1/2)=0$, the equilibrium phase is $\phi_e=c_e=1/2$, and $F'''(1/2)=0$. The formula for $\Phi''$ then gives $\Phi''(\theta_e)=\lambda c_*''(\theta_e)/\delta_e=-2\Phi'(\theta_e)/\theta_e$. Hence $\kappa_e=-2(L_cc_*'+L_\phi\Phi')(\theta_e)/\theta_e$. The critical relation $\mathscr H'_{\alpha_{\mathrm{crit}}}(\theta_e)=0$ reads $\alpha_{\mathrm{crit}}(L_cc_*'+L_\phi\Phi')(\theta_e)=a$, which gives the value of $\kappa_e$. The other two formulas follow from \eqref{eq:theta2} and \eqref{eq:centre-dynamics}. Since $\varrho<0$, the attracting side is $y>0$.
\end{proof}

\subsection{Heat exchange with the environment}
\label{sec:exchange}

A thin printed sheet exchanges heat with the surrounding air through its faces. We describe this exchange by Newton's law, with the external source given by
\begin{equation}\label{eq:newton}
H_{\mathrm{ext}}=-\chi\,(\theta-\theta_a)+Q,
\end{equation}
where $\theta_a>0$ is the ambient temperature. For a sheet of thickness $e_s$ with surface heat transfer coefficient $h_s$ on its two faces and absorbed radiative flux $q_r$, the effective volumetric coefficients are $\chi=2h_s/e_s>0$ and $Q=q_r/e_s\ge0$. The constants $\chi$, $\theta_a$ and $Q$ do not depend on $x$ and $t$. The homogeneous Neumann conditions are kept on $\partial\Omega$.
The additional term in \eqref{eq:newton} is of order zero, so the proof of Theorem~\ref{thm:localwp} applies without change.

A homogeneous equilibrium satisfies $\chi(\theta_e-\theta_a)=Q$, $c_e=c_*(\theta_e)$ and $F'(\phi_e)+\lambda\phi_e=\lambda c_e$. For $\lambda>1/4$ the last equation has a unique solution, because $s\mapsto F'(s)+\lambda s$ is strictly increasing. Hence the equilibrium is unique and given by
\begin{equation}\label{eq:newton-eq}
\theta_e=\theta_a+\frac{Q}{\chi},\qquad c_e=c_*(\theta_e),\qquad \phi_e=\Phi(\theta_e).
\end{equation}
It is fixed by the environment and does not depend on the initial data. It is also the only stationary solution, homogeneous or not. A stationary solution satisfies $-\nabla\cdot(k(\phi)\nabla\theta)+\chi(\theta-\theta_e)=0$, since the time derivatives vanish. Testing by $\theta-\theta_e$ gives
\[
\int_\Omega k(\phi)|\nabla\theta|^2\dx+\chi\int_\Omega|\theta-\theta_e|^2\dx=0,
\]
hence $\theta=\theta_e$. The chemical equation becomes $-d_c\Delta c+b_e(c-c_e)=0$, with $b_e=\beta K_d(\theta_e)+\gamma K_r(\theta_e)>0$, and testing by $c-c_e$ gives $c=c_e$. Finally, testing the phase equation by $\phi-\phi_e$ gives
\[
\varepsilon^2\int_\Omega|\nabla(\phi-\phi_e)|^2\dx+\int_\Omega\bigl(F'(\phi)+\lambda\phi-F'(\phi_e)-\lambda\phi_e\bigr)(\phi-\phi_e)\dx=0,
\]
and the strict monotonicity of $s\mapsto F'(s)+\lambda s$ gives $\phi=\phi_e$.
We keep the notation of Section~\ref{sec:exact-alpha-threshold} at this equilibrium and set, for $q=\Lambda_j$,
\[
\begin{aligned}
\kappa_q&=C_q+\chi, & U_q&=a(g_q+H_q)+\kappa_q,\\
V_q&=ag_qH_q+\kappa_q(g_q+H_q), & M_q&=L_cH_q+L_\phi\ell .
\end{aligned}
\]
For $x\ge0$ we define the polynomial
\begin{equation}\label{eq:newton-fq}
f_q(x)=(U_q-L_cx)(V_q-M_qx)-a\kappa_qg_qH_q .
\end{equation}

\begin{theorem}[Exact stability threshold under heat exchange]
\label{thm:exchange}
Assume \textup{\textbf{(H1)}}, $\lambda>1/4$, $\chi>0$, \eqref{eq:newton} and the regularity hypotheses of Theorem~\ref{thm:localwp} (with the external source replaced by \eqref{eq:newton}). Let $(\theta_e,c_e,\phi_e)$ be given by \eqref{eq:newton-eq}.
\begin{enumerate}[label=\textup{(\roman*)}]
\item The eigenvalues of the linearisation $\mathcal L_e^\chi$ at $(\theta_e,c_e,\phi_e)$ are the roots of the polynomials
\begin{equation}\label{eq:newton-poly}
P_q^\chi(s)=(as+\kappa_q)(s+g_q)(s+H_q)-\alpha rs\bigl[L_c(s+H_q)+L_\phi\ell\bigr],\quad\text{for } q=\Lambda_j .
\end{equation}
In particular $0\notin\sigma(\mathcal L_e^\chi)$.
\item If $r\le0$, or if $r>0$ and $L_c=L_\phi=0$, the equilibrium is locally exponentially stable in $X_\gamma$ for every $\alpha\ge0$.
\item Let $r>0$ and $L_c+L_\phi>0$. For every $q=\Lambda_j$, the polynomial $f_q$ has a smallest positive root $x_q$, and $x_q\to+\infty$ as $q\to+\infty$. The threshold defined by
\begin{equation}\label{eq:newton-exact}
\alpha_*(\chi)=\frac1r\min_{j\ge0}x_{\Lambda_j}
\end{equation}
is attained by finitely many modes and satisfies
\begin{equation}\label{eq:newton-bounds}
\alpha_{\mathrm{crit}}<\alpha_*(\chi)\le\alpha_H(\chi):=\frac{x_0}{r}.
\end{equation}
The spectral abscissa of $\mathcal L_e^\chi$ is negative if and only if $0\le\alpha<\alpha_*(\chi)$. For these values the equilibrium is locally exponentially stable in $X_\gamma$.
\item Let $r>0$ and $L_c+L_\phi>0$, and let $q_*$ be a mode attaining the minimum in \eqref{eq:newton-exact}. For every $\alpha>\alpha_*(\chi)$, the polynomial $P_{q_*}^\chi$ has a root with positive real part, so the spectral abscissa of $\mathcal L_e^\chi$ is positive. At $\alpha=\alpha_*(\chi)$ it has the purely imaginary roots $\pm i\omega_*$, with
\begin{equation}\label{eq:hopf-frequency}
\omega_*^2=\frac{V_{q_*}-M_{q_*}x_{q_*}}{a}>0 .
\end{equation}
These roots cross the imaginary axis transversally, with
\begin{equation}\label{eq:transversality}
\frac{\mathrm d}{\mathrm d\alpha}\operatorname{Re} s\Big|_{\alpha=\alpha_*(\chi)}=r\,\frac{aL_c\omega_*^2+M_{q_*}B_{2,q_*}}{2\bigl(a^2\omega_*^2+B_{2,q_*}^2\bigr)}>0,
\end{equation}
where $B_{2,q_*}=U_{q_*}-L_cx_{q_*}>0$. For every $\alpha>\alpha_*(\chi)$ the equilibrium is nonlinearly unstable in $X_\gamma$. If $q_*=0$, this follows also from the reduction to homogeneous solutions.
\item Let $r>0$ and $L_c+L_\phi>0$. Fix the equilibrium $(\theta_e,c_e,\phi_e)$ and let $\chi$ vary with $Q=\chi(\theta_e-\theta_a)$, so that $a,b,r,H_0,\ell$ do not depend on $\chi$. Then $\alpha_H(\chi)\to\alpha_{\mathrm{crit}}$ as $\chi\to0^+$ and $\alpha_H(\chi)\to+\infty$ as $\chi\to+\infty$.
\end{enumerate}
\end{theorem}

\begin{proof}
\emph{(i)} The linearisation of \eqref{eq:newton} adds $-\chi u$ to the right-hand side of the linearised heat equation, where $u$ is the temperature perturbation. In the modal computation of Theorem~\ref{thm:exact-alpha-threshold}, this replaces $as+C_q$ by $as+\kappa_q$, which gives \eqref{eq:newton-poly}. Writing $P_q^\chi(s)=as^3+B_2^\chi s^2+B_1^\chi s+B_0^\chi$ and $x=r\alpha$, one finds
\[
B_2^\chi=U_q-L_cx,\qquad B_1^\chi=V_q-M_qx,\qquad B_0^\chi=\kappa_qg_qH_q\ge\chi bH_0>0 .
\]
Hence $B_2^\chi B_1^\chi-aB_0^\chi=f_q(x)$, and $P_q^\chi(0)\ne0$ for every $q$.

\emph{A criterion for each mode.} By the Routh--Hurwitz criterion, all roots of $P_q^\chi$ have negative real part if and only if $B_2^\chi>0$, $B_1^\chi>0$ and $f_q(x)>0$. Conversely, a real cubic with positive leading coefficient whose roots lie in the closed left half-plane, none of them zero, has positive coefficients and $B_2B_1-aB_0\ge0$. Since $B_0^\chi>0$, the polynomial $P_q^\chi$ has a root with positive real part as soon as $f_q(x)<0$ or $B_1^\chi\le0$.

\emph{(ii)} If $r\le0$, then $x\le0$ and $B_2^\chi\ge U_q$, $B_1^\chi\ge V_q$. Thus $f_q(x)\ge U_qV_q-a\kappa_qg_qH_q$, and
\begin{equation}\label{eq:newton-identity}
U_qV_q-a\kappa_qg_qH_q=a^2(g_q+H_q)g_qH_q+a\kappa_q(g_q+H_q)^2+\kappa_q^2(g_q+H_q)>0 .
\end{equation}
If $L_c=L_\phi=0$, the polynomial $P_q^\chi$ does not depend on $\alpha$ and the case $\alpha=0$ applies. In both cases every root has negative real part. As in the proof of Theorem~\ref{thm:exact-alpha-threshold}, the real parts of the roots of $P_q^\chi$ tend to $-\infty$ as $q\to\infty$, uniformly for $\alpha$ in bounded sets, so the spectral abscissa is negative. Since $0$ is not an eigenvalue, the equilibrium is isolated and normally stable with trivial kernel. Theorem~\ref{thm:localwp} and \cite[Theorem~2.1]{PrussSimonettZacher2009} give local exponential stability in $X_\gamma$.

\emph{(iii)} Let $r>0$ and $L_c+L_\phi>0$. Then $M_q>0$. By \eqref{eq:newton-identity}, $f_q(0)>0$. At $\bar x_q=V_q/M_q$ one has $f_q(\bar x_q)=-a\kappa_qg_qH_q<0$. Thus $f_q$ has a smallest positive root $x_q\in\,(0,\bar x_q)$. If $L_c>0$, $f_q$ is a quadratic polynomial with positive leading coefficient, and $\bar x_q$ lies strictly between its roots. If $L_c=0$, $f_q$ is linear with negative slope. In both cases $f_q>0$ on $[0,x_q)$ and $f_q<0$ on $(x_q,\bar x_q]$.

On $[0,x_q)$ one has $B_1^\chi>0$, since $x<\bar x_q$. If $B_2^\chi$ vanished at some $x\in[0,x_q)$, then $f_q(x)=-a\kappa_qg_qH_q<0$, which is impossible. Hence $B_2^\chi>0$ on $[0,x_q)$, and all roots of $P_q^\chi$ have negative real part exactly for $0\le r\alpha<x_q$, among the values $r\alpha\le\bar x_q$.

For numerical evaluation, a formula valid also when $L_c=0$ is given by
\[
x_q=\frac{2D_q}{J_q+\sqrt{(L_cV_q-M_qU_q)^2+4aL_cM_q\kappa_qg_qH_q}}.
\]
Here $D_q=U_qV_q-a\kappa_qg_qH_q>0$ and $J_q=L_cV_q+M_qU_q>0$. This is the rationalised quadratic formula, or the root of the linear polynomial when $L_c=0$.

The coefficients $g_q$, $H_q$ and $\kappa_q$ are affine in $q$ with positive slopes $d_c$, $\varepsilon^2/\tau_\phi$ and $k_e$. Since $L_cM_qx^2\ge0$, one has
\[
f_q(x)\ge U_qV_q-a\kappa_qg_qH_q-x\bigl(L_cV_q+M_qU_q\bigr),
\]
and therefore, by \eqref{eq:newton-identity},
\[
x_q\ge\frac{a^2(g_q+H_q)g_qH_q}{L_cV_q+M_qU_q}.
\]
The numerator is a polynomial of degree three in $q$ with positive leading coefficient, and the denominator has degree at most two. Hence $x_q\to+\infty$, the minimum in \eqref{eq:newton-exact} is attained, and only finitely many modes attain it.

For $\alpha<\alpha_*(\chi)$, every mode satisfies $r\alpha<x_q$ and is strictly stable. Only finitely many modes have $x_q\le r\alpha+1$, and the others are uniformly stable as in the proof of (ii). The spectral abscissa is therefore negative. Statement (iv) shows that it is nonnegative for $\alpha\ge\alpha_*(\chi)$. Normal stability with trivial kernel and \cite[Theorem~2.1]{PrussSimonettZacher2009} give local exponential stability for $\alpha<\alpha_*(\chi)$.

For the lower bound in \eqref{eq:newton-bounds}, let $0\le\alpha\le\alpha_{\mathrm{crit}}$. The proof of Theorem~\ref{thm:exact-alpha-threshold} uses $C_q$ only through $C_q\ge0$, so it applies to $\kappa_q$ and gives $B_2^\chi>0$, $B_1^\chi>0$ and
\[
f_q(r\alpha)\ge a^2d_cqH_q^2+a\kappa_qH_q^2+a\kappa_qd_cqH_q+\kappa_q^2(g_q+H_q)\ge a\chi H_0^2>0 .
\]
Hence $x_q>r\alpha_{\mathrm{crit}}$ for every $q$, and $\alpha_*(\chi)>\alpha_{\mathrm{crit}}$ because the minimum is attained. The upper bound holds because $\alpha_*(\chi)\le x_0/r$.

\emph{(iv)} At $x=x_{q_*}$ one has $B_1^\chi=V_{q_*}-M_{q_*}x_{q_*}>0$ and $f_{q_*}=0$. A cubic $as^3+B_2s^2+B_1s+B_0$ with $B_1>0$ and $B_2B_1=aB_0$ factors as $(as+B_2)(s^2+B_1/a)$. This gives the roots $\pm i\omega_*$ with \eqref{eq:hopf-frequency}. Let $x=r\alpha>x_{q_*}$. If $x<\bar x_{q_*}$, then $f_{q_*}(x)<0$. If $x\ge\bar x_{q_*}$, then $B_1^\chi\le0$. In both cases $P_{q_*}^\chi$ has a root with positive real part, by the criterion established after (i). Hence the spectral abscissa is positive for every $\alpha>\alpha_*(\chi)$.

For the transversality, write $P(s,x)=as^3+(U_{q_*}-L_cx)s^2+(V_{q_*}-M_{q_*}x)s+\kappa_{q_*}g_{q_*}H_{q_*}$. At $x=x_{q_*}$ and $s=i\omega_*$ one has $B_1^\chi=a\omega_*^2$, so that
\[
\begin{aligned}
\partial_xP&=L_c\omega_*^2-iM_{q_*}\omega_*,\\
\partial_sP&=-3a\omega_*^2+2iB_{2,q_*}\omega_*+B_1^\chi=2\omega_*\bigl(-a\omega_*+iB_{2,q_*}\bigr).
\end{aligned}
\]
Since $\partial_sP\neq0$, the root $s(x)$ is a smooth function of $x$ near $x_{q_*}$, and
\[
\frac{\mathrm ds}{\mathrm dx}=-\frac{\partial_xP}{\partial_sP}=\frac{L_c\omega_*-iM_{q_*}}{2\bigl(a\omega_*-iB_{2,q_*}\bigr)}
=\frac{(L_c\omega_*-iM_{q_*})(a\omega_*+iB_{2,q_*})}{2\bigl(a^2\omega_*^2+B_{2,q_*}^2\bigr)}.
\]
Its real part is $(aL_c\omega_*^2+M_{q_*}B_{2,q_*})/\bigl(2(a^2\omega_*^2+B_{2,q_*}^2)\bigr)$. Since $x=r\alpha$, this gives \eqref{eq:transversality}. The positivity of $B_{2,q_*}$ was established in the proof of (iii).

For the nonlinear instability, let $\alpha>\alpha_*(\chi)$. By the first part of (iv), the spectrum of $\mathcal L_e^\chi$ contains an eigenvalue with positive real part. The spectrum of $\mathcal L_e^\chi$ consists of isolated eigenvalues, so its part in the open right half-plane is a nonempty finite set separated from the imaginary axis. The principle of linearised instability for quasilinear parabolic equations with maximal $L^p$-regularity~\cite[Section~5.4]{PrussSimonett2016} then gives nonlinear instability in $X_\gamma$.

For $q=0$ and $x>x_0$, either $x<\bar x_0$ and $f_0(x)<0$, or $x\ge\bar x_0$ and $B_1^\chi\le0$. In both cases $P_0^\chi$ has a root with positive real part. Spatially homogeneous solutions solve the system
\begin{equation}\label{eq:newton-ode}
\left\{\begin{aligned}
\dot h&=-\chi(\theta-\theta_a)+Q,\\
\dot c&=\mathcal R(\theta,c),\\
\tau_\phi\dot\phi&=\lambda(c-\phi)-F'(\phi),
\end{aligned}\right.
\qquad\text{with } \theta=\frac{h+\alpha L_cc+\alpha L_\phi\phi}{a}.
\end{equation}
Its Jacobian at the equilibrium is the restriction of $\mathcal L_e^\chi$ to constant functions, and its characteristic polynomial is a nonzero multiple of $P_0^\chi$. For $\alpha>\alpha_H(\chi)$ it has an eigenvalue with positive real part, so the equilibrium of \eqref{eq:newton-ode} is nonlinearly unstable. The class of homogeneous solutions is invariant under the flow of the partial differential system. The equilibrium is therefore nonlinearly unstable in $X_\gamma$.

\emph{(v)} The quantities $a,b,r,H_0,\ell,L_c,L_\phi$ are fixed, and $x_0$ is the smallest positive root of $f_0$. As $\chi\to0^+$, $f_0(x)\to(a(b+H_0)-L_cx)(abH_0-M_0x)$. Its smallest positive root is $r\alpha_{\mathrm{crit}}$, because $a(b+H_0)-L_cr\alpha_{\mathrm{crit}}\ge aH_0>0$. The roots of $f_0$ depend continuously on $\chi$, so $\alpha_H(\chi)\to\alpha_{\mathrm{crit}}$. Let $X>0$. Both factors of $f_0$ are nonincreasing in $x$. For $x\in[0,X]$ and $
\chi>\max\left\{
L_cX,\frac{M_0X}{b+H_0}
\right\}
$, they satisfy
\[
U_0-L_cx\ge\chi-L_cX>0,\qquad V_0-M_0x\ge\chi(b+H_0)-M_0X>0 .
\]
Hence one has
\[
f_0(x)\ge(\chi-L_cX)\bigl(\chi(b+H_0)-M_0X\bigr)-a\chi bH_0,\quad\text{for } x\in[0,X],
\]
and the right-hand side is positive for $\chi$ large enough. Thus $x_0>X$ for such $\chi$, and $\alpha_H(\chi)\to+\infty$.
\end{proof}

\begin{remark}[Selection of a nonhomogeneous mode]
\label{rem:mode-selection}
In the adiabatic case the first loss of stability always occurs in the homogeneous mode. With heat exchange this is no longer true. Let $\Omega=\,(0,\pi)$, so that $\Lambda_j=j^2$, and consider the following dimensionless coefficients,
\[
\begin{gathered}
a=1,\quad b=100,\quad r=1,\quad k_e=d_c=0.1,\quad \tau_\phi=0.01,\\
\varepsilon^2=0.5,\quad L_c=5,\quad L_\phi=10,\quad \chi=100,
\end{gathered}
\]
at the equilibrium $c_e=\phi_e=1/2$ with $\lambda=0.26$. Then $\delta_e=0.01$, $H_0=1$ and $\ell=26$. These values are realised by Arrhenius rates with $\beta K_d(\theta_e)=\gamma K_r(\theta_e)=50$ and $E_r-E_d=R\theta_e^2/25$. The smallest positive roots of the quadratic polynomials \eqref{eq:newton-fq} give the values
\[
\frac{x_0}{r}=36.47,\qquad \frac{x_1}{r}=29.65,\qquad \frac{x_4}{r}=33.10,\qquad \frac{x_9}{r}=36.36,
\]
rounded to two decimals. For $q\ge16$ an exact computation gives
\[
f_q(37)=\frac{251001q^3+18912095q^2-233933900q-1840000}{500}>0,
\]
together with $U_q>37L_c$ and $V_q>37M_q$. The inequalities ensure that both Routh--Hurwitz coefficients are positive at $x=37$ and that $f_q(37)>0$. The modal stability criterion in Theorem~\ref{thm:exchange} therefore gives $x_q>37$ for $q\ge16$. Hence $\alpha_*(\chi)=x_1/r\approx29.65<\alpha_H(\chi)$. For $x_1/r<\alpha<x_0/r$, the equilibrium is stable with respect to homogeneous perturbations and spectrally unstable with respect to the first nonhomogeneous mode. The heat exchange thus selects the first Neumann mode $\cos x$ before any homogeneous oscillation.
\end{remark}

\begin{remark}
Since $0$ is never an eigenvalue, the stability is lost at $\alpha_*(\chi)$ through at least one pair of purely imaginary eigenvalues. Write $\mathcal Q_*$ for the set of distinct Neumann eigenvalues attaining the minimum in \eqref{eq:newton-exact} and let $m_q$ denote the Neumann multiplicity of $q$. Each $q\in\mathcal Q_*$ contributes a simple pair $\pm i\omega_q$ in its three-dimensional modal matrix, with $\omega_q^2=(V_q-M_qx_q)/a$. The frequencies need not coincide. For any positive critical frequency $\omega$, the algebraic and geometric multiplicities of $i\omega$ in $\mathcal L_e^\chi$ both equal
\[
\sum_{\substack{q\in\mathcal Q_*\\\omega_q=\omega}}m_q.
\]
The modal factorisation in the proof of Theorem~\ref{thm:exchange} gives simple imaginary roots, and expansion in the Neumann basis also applies to generalised eigenvectors. Each modal pair crosses the imaginary axis transversally by \eqref{eq:transversality}. These are spectral conditions for oscillatory instability. A nondegenerate simple Hopf bifurcation would require a unique simple critical pair and a nonzero first Lyapunov coefficient. Multiple critical pairs require a separate analysis.
\end{remark}

\begin{proposition}[Dependence of the threshold on the heat exchange]
\label{prop:chi-dependence}
Assume the hypotheses of Theorem~\ref{thm:exchange}\textup{(iii)}, and let the equilibrium and the coefficients $a,b,r,H_0,\ell$ be fixed while $\chi$ varies, with $Q=\chi(\theta_e-\theta_a)$.
\begin{enumerate}[label=\textup{(\roman*)}]
\item The map $\chi\mapsto\alpha_*(\chi)$ is continuous on $(0,+\infty)$, and $\alpha_*(\chi)>\alpha_{\mathrm{crit}}$.
\item $\alpha_*(\chi)\to\alpha_{\mathrm{crit}}$ as $\chi\to0^+$. There exists $\chi_0>0$ such that, for $0<\chi<\chi_0$, the minimum in \eqref{eq:newton-exact} is attained only by $q=0$. In particular $\alpha_*(\chi)=\alpha_H(\chi)$, and the first instability is homogeneous.
\item There exists a constant $c_1>0$, depending only on the fixed equilibrium and the fixed coefficients but independent of $q$ and $\chi$, such that
\begin{equation}\label{eq:uniform-exchange-mode}
x_q(\chi)\ge c_1(1+q+\chi),\quad\text{for } q\ge0\text{ and }\chi>0.
\end{equation}
In particular, for suitable constants $c_2,C_2>0$,
\begin{equation}\label{eq:uniform-exchange-threshold}
c_2(1+\chi)\le\alpha_*(\chi)\le C_2(1+\chi),\quad\text{for } \chi>0.
\end{equation}
Thus $\alpha_*(\chi)\to+\infty$ as $\chi\to+\infty$: for every fixed $\alpha$, sufficiently strong heat exchange stabilises the equilibrium locally, simultaneously on all Neumann modes.
\item For every $q=\Lambda_j$, the map $\chi\mapsto x_q(\chi)$ is strictly increasing on $(0,+\infty)$. Consequently, $\chi\mapsto\alpha_*(\chi)$ is strictly increasing on $(0,+\infty)$.
\end{enumerate}
\end{proposition}

\begin{proof}
For $\chi$ in a compact set $[\chi_1,\chi_2]\subset\,(0,+\infty)$, the lower bound for $x_q$ obtained in the proof of Theorem~\ref{thm:exchange}\textup{(iii)} is a quotient of polynomials in $q$ and $\chi$, with numerator of degree three in $q$ and denominator of degree at most two. It tends to $+\infty$ as $q\to\infty$, uniformly for $\chi\in[\chi_1,\chi_2]$. Hence only finitely many modes can attain the minimum on $[\chi_1,\chi_2]$. Each $x_q$ is a simple root of $f_q$, since $f_q$ changes sign at $x_q$, so it depends continuously on $\chi$. A minimum of finitely many continuous functions is continuous. The inequality $\alpha_*(\chi)>\alpha_{\mathrm{crit}}$ is part of Theorem~\ref{thm:exchange}\textup{(iii)}.

For (ii), the polynomials $f_q$ are also defined for $\chi=0$. For $q\ge\Lambda_1>0$ one has $\kappa_q=k_eq+\chi\ge k_eq>0$, and the inequality
\[
f_q(r\alpha_{\mathrm{crit}})\ge a^2d_cqH_q^2+a\kappa_qH_q^2+a\kappa_qd_cqH_q+\kappa_q^2(g_q+H_q)
\]
from the proof of Theorem~\ref{thm:exchange}\textup{(iii)} remains valid for $\chi\in[0,1]$. Its right-hand side is bounded below by a positive constant independent of $q\ge\Lambda_1$ and $\chi\in[0,1]$. Together with the uniform growth of $x_q$ in $q$, this gives $\delta_0>0$ with $x_q\ge r\alpha_{\mathrm{crit}}+\delta_0$ for every $q\ge\Lambda_1$ and $\chi\in[0,1]$. On the other hand, $x_0\to r\alpha_{\mathrm{crit}}$ as $\chi\to0^+$, by Theorem~\ref{thm:exchange}\textup{(v)}. Hence $x_0<x_q$ for every $q\ge\Lambda_1$ when $\chi$ is small, which proves (ii).

For (iii), expand $f_q$ at its first positive root. Since $L_cM_qx_q^2\ge0$,
\[
x_q\ge \frac{U_qV_q-a\kappa_qg_qH_q}{L_cV_q+M_qU_q}.
\]
Let $\Sigma=1+q$ and $\Upsilon=1+q+\chi$. The positive affine functions $g_q,H_q$ are bounded above and below by positive multiples of $\Sigma$, uniformly for $q\ge0$. Also $\kappa_q\le C\Upsilon$ and $\Sigma+\kappa_q\ge c\Upsilon$. All constants here depend only on fixed coefficients. By \eqref{eq:newton-identity}, retaining the first and last terms gives
\[
U_qV_q-a\kappa_qg_qH_q\ge c\,\Sigma(\Sigma^2+\kappa_q^2)\ge c'\,\Sigma\Upsilon^2.
\]
On the other hand, $U_q\le C\Upsilon$, $V_q\le C\Sigma\Upsilon$ and $0<M_q\le C\Sigma$, so $L_cV_q+M_qU_q\le C'\Sigma\Upsilon$. Division proves \eqref{eq:uniform-exchange-mode}. Taking the minimum over the Neumann spectrum and dividing by $r$ gives the lower bound in \eqref{eq:uniform-exchange-threshold}. Since $x_0<V_0/M_0$, with $M_0>0$ fixed and $V_0$ affine in $\chi$, the upper bound follows from $\alpha_*(\chi)\le x_0/r$. The nonlinear stability assertion is a direct application of Theorem~\ref{thm:exchange}\textup{(iii)}.

For (iv), fix $q$ and write $g,H,\kappa,U,V,M$ for $g_q,H_q,\kappa_q,U_q,V_q,M_q$. Since $\partial_\chi\kappa=1$, $\partial_\chi U=1$ and $\partial_\chi V=g+H$, differentiation of \eqref{eq:newton-fq} at fixed $x$ gives
\[
\partial_\chi f_q(x)=(V-Mx)+(g+H)(U-L_cx)-agH.
\]
Set $B_2=U-L_cx_q>0$ and $B_1=V-Mx_q>0$, as in the proof of Theorem~\ref{thm:exchange}\textup{(iii)}. At the root, $B_2B_1=a\kappa gH$, and $B_1=agH+\kappa(g+H)-Mx_q$. Hence
\[
\partial_\chi f_q(x_q)=B_1+\frac{B_2}{\kappa}\,\bigl(Mx_q-agH\bigr).
\]
We show that $x_q>\hat x=agH/M$. Since $M\ge L_cH$, one has $aL_cgH/M\le ag$, so that $U-L_c\hat x\ge aH+\kappa$. Moreover $V-M\hat x=\kappa(g+H)$. Therefore
\[
f_q(\hat x)\ge(aH+\kappa)\,\kappa\,(g+H)-a\kappa gH>0.
\]
Since $\hat x<\bar x_q$ and $f_q<0$ on $(x_q,\bar x_q]$, this gives $\hat x<x_q$. Thus $\partial_\chi f_q(x_q)>0$. The root $x_q$ is simple and $\partial_xf_q(x_q)<0$, because $f_q>0$ on $[0,x_q)$ and $f_q<0$ on $(x_q,\bar x_q]$. The implicit function theorem gives
\[
\frac{\mathrm dx_q}{\mathrm d\chi}=-\frac{\partial_\chi f_q(x_q)}{\partial_xf_q(x_q)}>0.
\]
Let $0<\chi_1<\chi_2$ and let $q_2$ attain the minimum in \eqref{eq:newton-exact} at $\chi_2$. Then $r\alpha_*(\chi_2)=x_{q_2}(\chi_2)>x_{q_2}(\chi_1)\ge r\alpha_*(\chi_1)$.
\end{proof}

\begin{proposition}[Mode selection for strong heat exchange]
\label{prop:large-chi}
Assume the hypotheses of Proposition~\ref{prop:chi-dependence}, with $L_c>0$ and $L_cb>L_\phi\ell$. For $q=\Lambda_j$, define
\[
\Phi_q=\frac{a(g_q+H_q)+C_q}{L_c}-\frac{a\,g_qH_q}{L_cg_q-L_\phi\ell}.
\]
Then, for every fixed $q=\Lambda_j$, the threshold of the mode $q$ satisfies
\begin{equation}\label{eq:large-chi}
x_q(\chi)=\frac{\chi}{L_c}+\Phi_q+O\Bigl(\frac1\chi\Bigr),\quad\text{as }\chi\to+\infty.
\end{equation}
Consequently, if $\Phi_q<\Phi_0$ for some $q=\Lambda_j$ with $j\ge1$, there exists $\chi_1>0$ such that, for every $\chi>\chi_1$, the minimum in \eqref{eq:newton-exact} is not attained by $q=0$. The first instability is then nonhomogeneous.
\end{proposition}

\begin{proof}
Fix $q=\Lambda_j$ and write $g,H,C,U,V,M,\kappa$ for $g_q,H_q,C_q,U_q,V_q,M_q,\kappa_q$. Since $g\ge b$, one has $L_cg-L_\phi\ell\ge L_cb-L_\phi\ell>0$. Set $x_A=U/L_c$. Then $f_q(x_A)=-a\kappa gH<0$ and $f_q(0)>0$. The quadratic polynomial $f_q$ has the positive leading coefficient $L_cM$, so it has exactly one root in $(0,x_A)$ and its other root is larger than $x_A$. Hence $x_q\in(0,x_A)$. Put $y=x_A-x_q>0$ and $D=V-Mx_A$. Since $U-L_cx_q=L_cy$ and $V-Mx_q=D+My$, the equation $f_q(x_q)=0$ reads
\[
L_cM\,y^2+L_cD\,y-a\kappa gH=0.
\]
A direct computation with $M=L_cH+L_\phi\ell$ gives
\[
L_cD=L_cV-MU=\kappa\,(L_cg-L_\phi\ell)+a\bigl(L_cgH-M(g+H)\bigr).
\]
Since $\kappa=C+\chi$, one has $L_cD=\kappa(L_cg-L_\phi\ell)+O(1)>0$ for $\chi$ large. The positive root is given by
\[
y=\frac{2a\kappa gH}{L_cD+\sqrt{L_c^2D^2+4aL_cM\kappa gH}}.
\]
The square root equals $L_cD\,(1+O(1/\chi))$, hence
\[
y=\frac{a\,gH}{L_cg-L_\phi\ell}+O\Bigl(\frac1\chi\Bigr).
\]
Since $x_A=(a(g+H)+C+\chi)/L_c$, this proves \eqref{eq:large-chi}. If $\Phi_q<\Phi_0$, then \eqref{eq:large-chi} for $q$ and for $0$ gives $x_q(\chi)-x_0(\chi)=\Phi_q-\Phi_0+O(1/\chi)<0$ for $\chi$ large. Therefore $\min_jx_{\Lambda_j}\le x_q<x_0$, and $q=0$ does not attain the minimum.
\end{proof}

\begin{remark}
\label{rem:large-chi-example}
For the coefficients of Remark~\ref{rem:mode-selection}, with $\lambda\in(1/4,1/2)$, one has $L_cb=500>L_\phi\ell=1000\lambda$. With $u=1000\lambda$, a direct computation gives
\[
\Phi_1-\Phi_0=10.04-\frac{100.1\,(u/10+25)}{500.5-u}+\frac{100\,(u/10-25)}{500-u}.
\]
Multiplying by $(500.5-u)(500-u)>0$, the inequality $\Phi_1<\Phi_0$ is equivalent to
\[
-\frac{201}{20}\,u^2+\frac{126063}{25}\,u-10010>0.
\]
This concave quadratic polynomial is positive at $u=250$ and at $u=499$, hence on $[250,499]$. By Proposition~\ref{prop:large-chi}, for every $\lambda\in(1/4,0.499]$ the first instability is nonhomogeneous when the heat exchange is strong enough.
\end{remark}

\begin{remark}[Mode switching in the example]
\label{rem:mode-switching}
For the coefficients of Remark~\ref{rem:mode-selection}, the critical mode $q_*(\chi)$ takes increasing values at the four exchange coefficients of the table below. We have verified the following statements with exact rational arithmetic. For each pair $(\chi,X)$ below, $f_{q_*}(X)<0$, so $x_{q_*}<X$. For every other mode $q$, one has $f_q(X)>0$, $U_q>L_cX$ and $V_q>M_qX$. These three quantities are polynomials in $q$ of degrees three, one and two with rational coefficients. After the substitution $q=q_0+t$, with $q_0=1,4,9,25$ for the four values of $\chi$, all their coefficients are positive, so they are positive for $q\ge q_0$. The remaining modes $q<q_0$ are checked individually. The argument of Remark~\ref{rem:mode-selection} then gives $x_q>X$.
\[
\begin{array}{cccc}
\chi & X & q_* & x_{q_*}/r\\\hline
10 & 8.019 & 0 & 4.147\\
100 & 31.371 & 1 & 29.646\\
1000 & 202.282 & 4 & 200.757\\
10000 & 1963.138 & 16 & 1961.611
\end{array}
\]
For these four values the critical wavenumber increases with $\chi$. Numerical comparisons of the modal thresholds indicate switches from $q=0$ to $q=1$ near $\chi\approx53.9$, from $q=1$ to $q=4$ near $\chi\approx270.5$, from $q=4$ to $q=9$ near $\chi\approx1918.7$ and from $q=9$ to $q=16$ near $\chi\approx7579.1$. These approximate pairwise crossing values were obtained by root finding on the computed thresholds $x_q$. The four entries of the table are proved by the polynomial inequalities above. In this example a stronger heat exchange stabilises long wavelengths more efficiently and selects finer spatial patterns. The critical mode is not monotone in $\chi$ in general. For the same coefficients with $\lambda=0.57$, exact rational arithmetic gives, at $\chi=2500$ and $X=432.8$,
\[
f_1(X)=-\frac{69632979}{500}<0,\qquad f_0(X)=75808>0,
\]
with $U_0>L_cX$ and $V_0>M_0X$. At $\chi=3000$ and $X=521.8$, it gives
\[
f_0(X)=-36422<0,\qquad f_1(X)=\frac{14409074}{125}>0,
\]
with $U_1>L_cX$ and $V_1>M_1X$. In both cases, after the substitution $q=4+t$, the polynomials $U_q-L_cX$, $V_q-M_qX$ and $f_q(X)$ have positive coefficients. Hence $q_*(2500)=1$ and $q_*(3000)=0$.

The first crossing also illustrates why the critical frequencies must be distinguished. There exists $\widehat\chi\in(53.900,53.901)$ at which $x_0=x_1=\min_j x_{j^2}$. Rational bisection brackets the two roots in opposite order at the endpoints. To exclude $q\ge4$ throughout this interval, set $X=22$, $\kappa_q^-=k_eq+53.900$, $\kappa_q^+=k_eq+53.901$, and let $U_q^-,V_q^-$ be evaluated at $\kappa_q^-$. After substituting $q=4+t$, all coefficients of the three polynomials
\[
U_q^--L_cX,\qquad V_q^--M_qX,\qquad
(U_q^--L_cX)(V_q^--M_qX)-a\kappa_q^+g_qH_q
\]
are positive. Hence $x_q>22$ for $q\ge4$ throughout the interval. The same rational interval tests, now for $q=0,1$, give $20.5<x_q<20.6$. Continuity then yields the claimed simultaneous minimum. The squared critical frequencies $\omega_q^2=(V_q-M_qx_q)/a$ are bounded by their values at the endpoints $x\in\{20.5,20.6\}$ and $\chi\in\{53.900,53.901\}$. This gives $84.9\le\omega_0^2\le111.501$ and $2655.5\le\omega_1^2\le2707.1511$, so the two pairs are distinct. Numerically, $\widehat\chi\approx53.900313$, $\alpha_*(\widehat\chi)\approx20.531178$, $\omega_0\approx10.157239$ and $\omega_1\approx51.874760$. This is a simultaneous crossing of two distinct imaginary pairs, not a single pair of multiplicity two.
\end{remark}

	\subsection{Numerical illustration of the full system}
\label{sec:numerics}

We illustrate Theorems~\ref{thm:local-stability}, \ref{thm:bifurcation} and~\ref{thm:exchange} by simulations of the nonlinear system on $\Omega=\,(0,\pi)$. We use dimensionless units with $R=1$ and $\theta_e=1$, and the coefficients
\[
\begin{gathered}
a=1,\quad k\equiv0.1,\quad d_c=0.1,\quad \tau_\phi=0.01,\quad \varepsilon^2=0.5,\\
\lambda=0.26,\quad L_c=5,\quad L_\phi=10,\quad \beta=\gamma=1,\\
E_d=1,\quad E_r=1.04,\quad A_d=50\,e^{E_d},\quad A_r=50\,e^{E_r}.
\end{gathered}
\]
The point $(\theta_e,c_e,\phi_e)=(1,1/2,1/2)$ is then an equilibrium with $b=100$, $r=1$, $H_0=1$ and $\ell=26$, which are the values of Remark~\ref{rem:mode-selection}. With $\ell_0=\pi$, these values give $D_\theta=D_c=10^{-3}/\pi^2\approx1.01\times10^{-4}$, $T_\phi=1$, $D_\phi=0.5/\pi^2\approx0.051$ and $\theta_ec_*'(\theta_e)=0.01$. For $\chi=100$ one has $\bar\chi=1$, so that heat exchange and chemical relaxation act on the same time scale. Since only the products $\theta_ec_*'(\theta_e)\mathrm{Lt}_c$ and $\theta_ec_*'(\theta_e)\mathrm{Lt}_\phi$ enter the modal polynomials, the thresholds of this example coincide with those of any parameter set with $\theta_ec_*'(\theta_e)=1$ and latent heat ratios divided by $100$, for instance $\Gamma_r-\Gamma_d=4$ at $c_e=1/2$. At $\alpha=\alpha_*(100)\approx29.65$, this rescaled set has $\mathrm{Lt}_c\approx1.48$ and $\mathrm{Lt}_\phi\approx2.97$. The Laplacian is discretised by second-order centred differences on $N_h$ uniform nodes $x_i=i\pi/(N_h-1)$. The Neumann condition is imposed by reflection, $u_{-1}=u_1$ and $u_{N_h}=u_{N_h-2}$, which gives the boundary rows $2(u_1-u_0)/\Delta x^2$ and $2(u_{N_h-2}-u_{N_h-1})/\Delta x^2$. The semi-discrete system is integrated by the variable-order backward differentiation method. We integrate in perturbation enthalpy variables and supply the analytic sparse Jacobian. We use the relative and absolute tolerances $(10^{-9},10^{-12})$ for the heat-exchange experiments, $(10^{-12},10^{-14})$ for the homogeneous decay at $\alpha=25$, and $(10^{-10},10^{-13})$ for the isolated system. The signed modal coefficient of the temperature perturbation is its trapezoidal projection $\sum_i \varpi_i(\theta_i-\theta_e)\cos(jx_i)/\sum_i \varpi_i\cos^2(jx_i)$. The amplitude is its absolute value, with $\varpi_0=\varpi_{N_h-1}=1/2$ and $\varpi_i=1$ otherwise. Growth rates of oscillating modes are least-squares slopes of the logarithm of the successive local maxima of the amplitude on the stated time window (sample spacing $2.5\times10^{-4}$). Non-oscillating rates are least-squares slopes of the logarithm of the amplitude itself.

\emph{Heat exchange.} Let $\chi=100$, $\theta_a=1$ and $Q=0$, so that $\alpha_*(100)\approx29.65$ and $\alpha_H(100)\approx36.47$. For $\alpha=32$ the initial state $\theta=1+10^{-6}\cos x$, $c=\phi=1/2$ produces a growing oscillation of the mode $\cos x$, while a homogeneous perturbation of the same size decays. Table~\ref{tab:numerics} compares the growth rate of $\cos x$, measured on the successive maxima for $0.3\le t\le1.5$, with the real part of the root of $P_1^\chi$. The eigenvalues of the discrete Jacobian with largest real part are $5.887\pm70.20\,i$ for $N_h=65$ and $N_h=129$, in agreement with $P_1^\chi$. For $\alpha=25$ the equilibrium is locally exponentially stable by Theorem~\ref{thm:exchange}, and the homogeneous simulation decays at the rate reported in the table (fit on $0.5\le t\le1.5$).

\emph{Isolated system.} Let $\chi=0$, so that $\alpha_{\mathrm{crit}}=100/265\approx0.3774$. For $\alpha=0.30$ and $\alpha=0.45$ we take homogeneous initial data with $c-c_e=10^{-5}$, $\theta-\theta_e=\alpha L_c(c-c_e)/a$ and $\phi=\phi_e$, so the enthalpy is exactly that of the reference equilibrium. We measure the rate from $|c-c_e|$ on $2\le t\le10$. In the perturbation enthalpy variables the homogeneous enthalpy perturbation remains identically zero. The measured rates agree with the slow root of $as^2+T_\alpha s+D_\alpha$ (Table~\ref{tab:numerics}).

\emph{Critical threshold.} For $\alpha=\alpha_{\mathrm{crit}}$ we start in the enthalpy class of the equilibrium, with $\theta-\theta_e=2\times10^{-3}$, $c-c_e=a(\theta-\theta_e)/(\alpha L_c)$ and $\phi=\phi_e$. Corollary~\ref{cor:symmetric} gives $\varrho=-1.0089$ and predicts decay for a positive temperature shift. Figure~\ref{fig:numerics}(b) shows that $t(\bar\theta(t)-\theta_e)$ increases towards $1/|\varrho|=0.9911$, with the values $0.8345$, $0.9058$, $0.9464$ and $0.9682$ at $t=5\times10^3$, $10^4$, $2\times10^4$ and $4\times10^4$. The slow approach is consistent with the asymptotic remainder in Theorem~\ref{thm:bifurcation}\textup{(iii)}. The repelling side is characterised by the centre-manifold analysis of Theorem~\ref{thm:bifurcation}, and the plotted trajectory illustrates the attracting side. We repeated the positive shift with the additional perturbation $10^{-4}\cos x$ of the temperature. The nonconstant mode exhibits exponential decay over the considered time interval, at rate $0.1039$ (fit on $10\le t\le30$), in agreement with the slowest root $-0.1040$ of $P_1$ at $\alpha_{\mathrm{crit}}$, while the mean temperature follows the same algebraic relaxation, with $t(\bar\theta(t)-\theta_e)=0.906$ at $t=10^4$ and $0.968$ at $t=4\times10^4$.

\begin{table}[ht]
\centering
\footnotesize
\begin{tabular}{@{}p{4.6cm}cccc@{}}
\hline
Case & $N_h=33$ & $N_h=65$ & $N_h=129$ & Theory\\
\hline
$\chi=100$, $\alpha=32$, mode $\cos x$ & $5.8877$ & $5.8875$ & $5.8875$ & $5.8870$\\
$\chi=100$, $\alpha=25$, homogeneous mode & $-2.9784$ & $-2.9784$ & $-2.9784$ & $-2.9784$\\
$\chi=0$, $\alpha=0.30$, slow mode & \multicolumn{3}{c}{homogeneous ODE: $-0.2065$} & $-0.2065$\\
$\chi=0$, $\alpha=0.45$, slow mode & \multicolumn{3}{c}{homogeneous ODE: $0.1945$} & $0.1946$\\
\hline
\end{tabular}
\caption{Rates measured in the nonlinear simulations and rates predicted by the modal polynomials.}
\label{tab:numerics}
\end{table}

\begin{figure}[ht]
\centering
\includegraphics[width=\textwidth]{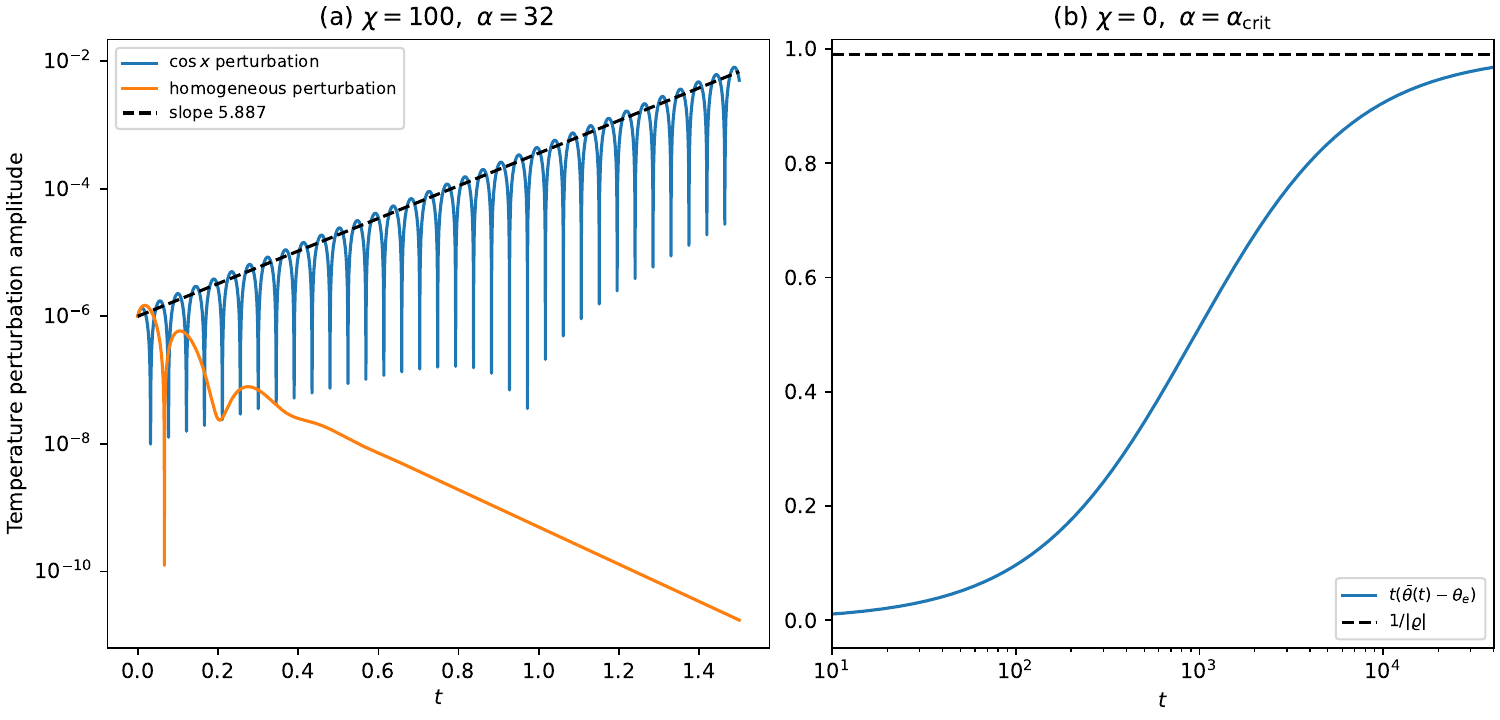}
\caption{(a) Temperature perturbation amplitudes from two separate simulations for $\chi=100$ and $\alpha=32$: a $\cos x$ perturbation ($N_h=65$) and a homogeneous perturbation, with the predicted growth slope. (b) Critical relaxation for $\chi=0$ and $\alpha=\alpha_{\mathrm{crit}}$ after a positive temperature shift in the enthalpy class of the equilibrium.}
\label{fig:numerics}
\end{figure}

\emph{Dependence on the phase and exchange coefficients.} Figure~\ref{fig:modemap} shows the critical wavenumber $\sqrt{q_*}$ in the plane $(\chi,\lambda)$. All other coefficients are those of the example, and $\delta_e=\lambda-1/4$ at $\phi_e=1/2$. The map is computed from the explicit roots $x_q$ on a logarithmic grid of $240\times240$ points, and it shows the minimiser among the modes $q=j^2$ with $0\le j\le200$. Analytical results on mode selection are given in Remarks~\ref{rem:mode-selection}, \ref{rem:large-chi-example} and~\ref{rem:mode-switching}. The figure shows a region of the parameter plane in which a nonhomogeneous mode is selected. Its lower boundary lies approximately at $\chi\approx51$ for $\lambda$ close to $1/4$, and it moves to larger $\chi$ as $\lambda$ increases. For $\lambda$ close to $1/4$ the critical wavenumber increases with $\chi$ and reaches $\sqrt{q_*}=4$ at $\chi=10^4$. Near the upper edge of the region the critical mode returns to $q=0$ as $\chi$ increases, see Remark~\ref{rem:mode-switching}. In the computed range the region lies below $\lambda\approx0.58$. A larger $\lambda$ increases the relaxation rate $H_0=(\lambda-1/4)/\tau_\phi$ of the homogeneous phase mode, and the homogeneous mode is then the first to lose stability. Remark~\ref{rem:large-chi-example} proves that, for every $\lambda\in(1/4,0.499]$, the first instability is nonhomogeneous for $\chi$ large enough. The selection of a nonhomogeneous mode therefore extends beyond the neighbourhood of $\lambda=1/4$, to a region of moderate chemical--phase coupling and strong heat exchange.

\begin{figure}[ht]
\centering
\includegraphics[width=0.8\textwidth]{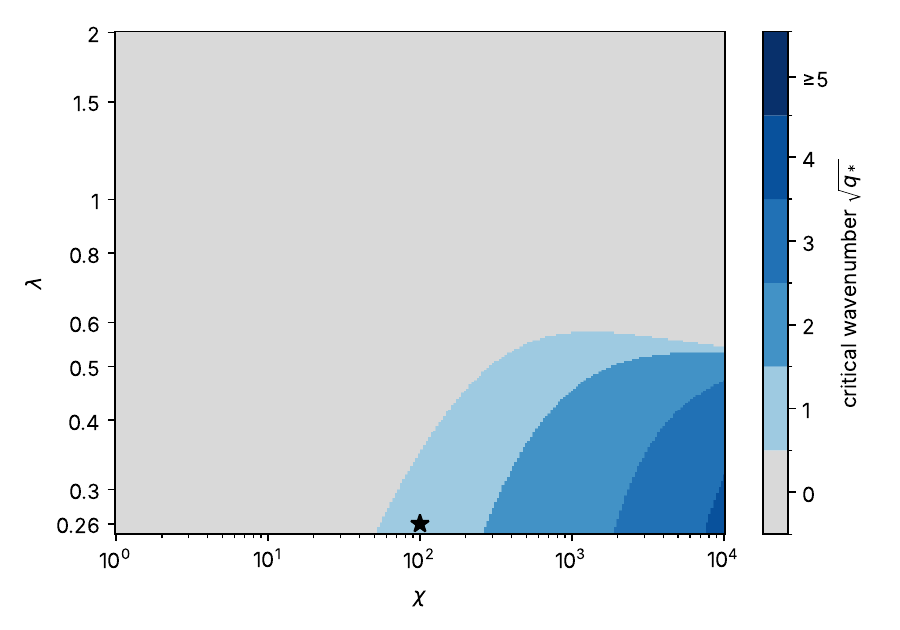}
\caption{Critical wavenumber $\sqrt{q_*}$ as a function of the heat exchange coefficient $\chi$ and of the chemical--phase coupling $\lambda$, for the coefficients of Remark~\ref{rem:mode-selection} with $\Omega=\,(0,\pi)$. The grey region corresponds to a homogeneous first instability. The star marks the example $\chi=100$, $\lambda=0.26$.}
\label{fig:modemap}
\end{figure}

\subsection{Physical interpretation}
\label{sec:interpretation}

The results translate into four statements on the colour equilibrium of the model. They hold near a homogeneous equilibrium. The reaction rates do not depend directly on $\phi$, and $c$ is not related to an optical observable. Quantitative comparison with time-resolved reflectance measurements for a given ink requires parameter identification and an observation model relating $c$ to the measured reflectance.

\emph{Robustness of direct thermochromism.}
In the model, a coloured fraction that decreases on heating
corresponds to $E_r<E_d$, hence $r<0$.
By Theorems~\ref{thm:local-stability} and~\ref{thm:exchange}\textup{(ii)}, the equilibrium is locally
stable for every fixed $\alpha\geq 0$.
In the adiabatic case, nearby solutions converge exponentially
to the equilibrium selected by their conserved enthalpy,
whereas, in the presence of heat exchange, they converge
exponentially to the unique equilibrium determined by the
external thermal conditions.

\emph{A design criterion for inverse thermochromism.} In the model, colour development on heating corresponds to $E_r>E_d$. In an insulated sample, the equilibrium loses stability when
\[
c_e(1-c_e)(\Gamma_r-\Gamma_d)\Bigl(\mathrm{Lt}_c+\frac{\lambda}{F''(\phi_e)+\lambda}\,\mathrm{Lt}_\phi\Bigr)
\]
reaches the value one. Within the model, equilibria in the inverse thermochromic regime are therefore more susceptible to instability when the latent heat coefficients are large compared with $\rho c_p\theta_e$, when the gap between the reduced activation energies is large, and when the equilibrium chemical fraction $c_e$ is close to $1/2$. Each factor could be estimated from calorimetric and kinetic measurements, see Table~\ref{tab:parameters}.
\newpage
\emph{Role of sheet thickness and the environment.} For a printed sheet, $\bar\chi=2h_s/(\rho c_pb_ee_s)$. At fixed equilibrium and with all other coefficients held fixed, Proposition~\ref{prop:chi-dependence} shows that the critical coupling increases strictly with $\chi$ and lies between two positive multiples of $1+\chi$. In the model, a thinner or better ventilated sheet can therefore sustain a larger thermal coupling before losing stability, provided the equilibrium is kept fixed. A change of ventilation also changes $\theta_e=\theta_a+Q/\chi$ in general, and the comparison then involves the coefficients at the new equilibrium. When stability is lost, the first unstable mode may be nonhomogeneous, and at the linear level the instability then appears as an oscillating spatial modulation of the chemical fraction, with spatial dependence given by a critical Neumann eigenfunction, rather than a spatially uniform perturbation. In the example of Figure~\ref{fig:modemap}, this occurs for moderate chemical--phase coupling and $\bar\chi$ above about one half. For $\lambda=0.26$ stronger exchange selects shorter wavelengths. Remark~\ref{rem:mode-switching} gives a value of $\lambda$ at which the critical mode returns to $q=0$.

\emph{Critical slowing down near the threshold.} Below the adiabatic threshold, small perturbations at fixed enthalpy decay exponentially. The decay rate of the linearised homogeneous mode tends to zero linearly at the threshold, by Proposition~\ref{prop:critical-slowing}. At the threshold, for homogeneous solutions on the attracting side of the centre manifold of the critical enthalpy class, the relaxation becomes algebraic, of order $1/t$, by Theorem~\ref{thm:bifurcation}\textup{(iii)}. This gives a quantitative prediction of the critical relaxation rate.

\section{Conclusion}
	The nonlinear stability and bifurcation results in Section~\ref{sec:local} concern a neighbourhood of a homogeneous equilibrium. In the isolated system, conserved enthalpy selects the limiting equilibrium. The equilibrium is stable for every fixed coupling if $E_r\le E_d$, and loses stability homogeneously at an explicit threshold if $E_r>E_d$ and $L_c+L_\phi>0$, through a transcritical bifurcation along the parameter-dependent enthalpy classes when $\kappa_e\ne0$. At fixed equilibrium, the threshold with heat exchange is larger than the adiabatic threshold, increases strictly with $\chi$, and lies between two positive multiples of $1+\chi$. The first unstable mode may be nonconstant. The thresholds depend on the coefficients through the latent heat ratios weighted by the reduced activation-energy gap, and the selection of a nonhomogeneous mode occupies a region of moderate chemical--phase coupling and strong heat exchange. In the model, equilibria in the direct thermochromic regime are robust, while those in the inverse regime satisfy an explicit stability criterion, and, at fixed equilibrium, thinner or better ventilated sheets have a higher stability threshold but may develop spatial modulations when stability is lost. The numerical experiments illustrate the modal growth rates and the critical relaxation.

The regime $\lambda>1/4$ has a unique stationary phase state at fixed chemical fraction. Extensions involving optical observables, phase-dependent chemical rates, temperature-dependent phase potentials and weaker phase coupling could address equilibrium colour hysteresis, as observed in leuco dye composites~\cite{KulcarEtal2010,BasnecEtal2018}.

\section*{Declaration of competing interest}
The authors declare that they have no known competing financial interests or personal relationships that could have appeared to influence the work reported in this paper.

\section*{Funding}
This research did not receive any specific grant from funding agencies in the public, commercial, or not-for-profit sectors.

\section*{Data availability}
No experimental data were used. The numerical results can be reproduced from the formulas and parameters given in the paper.

\end{document}